\documentclass[11pt]{article}
\usepackage[utf8]{inputenc}
\usepackage[top=2cm,bottom=2cm,left=2cm,right=2cm]{geometry}
\usepackage{amsthm,amsmath,amssymb, verbatim}
\usepackage{hyperref} 
\usepackage{enumitem,cleveref}
\usepackage{authblk} 
\usepackage{bm} 
\usepackage[normalem]{ulem} 
\usepackage{mathtools} 
\usepackage{tikz}
\usetikzlibrary {arrows.meta}

\usepackage{url}
\usepackage[width=.76\textwidth]{caption}

\newtheorem{theorem}{Theorem}
\newtheorem{lemma}[theorem]{Lemma}
\newtheorem{corollary}[theorem]{Corollary}

\DeclareMathOperator{\tr}{tr}
\DeclareMathOperator{\sgn}{sgn}
\DeclareMathOperator{\rank}{rank}
\DeclareMathOperator{\im}{im}

\title{Oscillations in three-reaction quadratic mass-action systems}
\author[1]{Murad Banaji}
\author[2]{Bal\'azs Boros\footnote{BB's work was supported by the Austrian Science Fund, project P32532.}}
\author[2]{Josef Hofbauer}
\affil[1]{Department of Design Engineering and Mathematics, Middlesex University London, UK}
\affil[2]{Department of Mathematics, University of Vienna, Austria}
\date{}

\begin{document}

\maketitle

\begin{abstract}
It is known that rank-two bimolecular mass-action systems do not admit limit cycles. With a view to understanding which small mass-action systems admit oscillation, in this paper we study rank-two networks with bimolecular source complexes but allow target complexes with higher molecularities. As our goal is to find oscillatory networks of minimal size, we focus on networks with three reactions, the minimum number that is required for oscillation. However, some of our intermediate results are valid in greater generality.

One key finding is that an isolated periodic orbit cannot occur in a three-reaction, trimolecular, mass-action system with bimolecular sources. In fact, we characterise all networks in this class that admit a periodic orbit; in every case all nearby orbits are periodic too. Apart from the well-known Lotka and Ivanova reactions, we identify another network in this class that admits a center. This new network exhibits a vertical Andronov--Hopf bifurcation.

Furthermore, we characterise all two-species, three-reaction, bimolecular-sourced networks that admit an Andronov--Hopf bifurcation with mass-action kinetics. These include two families of networks that admit a supercritical Andronov--Hopf bifurcation, and hence a stable limit cycle. These networks necessarily have a target complex with a molecularity of at least four, and it turns out that there are exactly four such networks that are tetramolecular.
\end{abstract}

\section{Introduction}

There are two well-known small reaction networks that exhibit oscillations: the Lotka reactions \cite{lotka:1920} and the Ivanova reactions \cite[page 630]{volpert:hudjaev:1985}. The networks, along with their associated mass-action differential equations, are
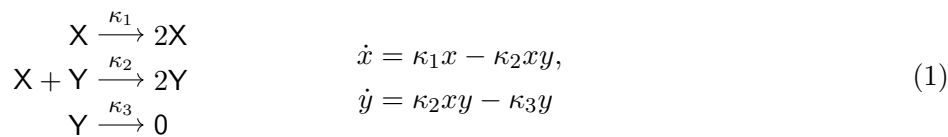
\begin{align}\label{eq:lotka_only}
\begin{aligned}
    \begin{tikzpicture}[scale=0.6]

\node[left]  (P1) at (0, 0) {$\mathsf{X}$};
\node[right] (P2) at (1, 0) {$2\mathsf{X}$};
\node[left]  (P3) at (0,-1) {$\mathsf{X+Y}$};
\node[right] (P4) at (1,-1) {$2\mathsf{Y}$};
\node[left]  (P5) at (0,-2) {$\mathsf{Y}$};
\node[right] (P6) at (1,-2) {$\mathsf{0}$};

\draw[->] (P1) to node[above] {\footnotesize $\kappa_1$} (P2);
\draw[->] (P3) to node[above] {\footnotesize $\kappa_2$} (P4);
\draw[->] (P5) to node[above] {\footnotesize $\kappa_3$} (P6);

\node at (8,-1)
{$\begin{aligned}
\dot{x} &= \kappa_1 x - \kappa_2 xy, \\
\dot{y} &= \kappa_2 xy - \kappa_3 y \\
\end{aligned}$};

\end{tikzpicture}
\end{aligned}
\end{align}
and
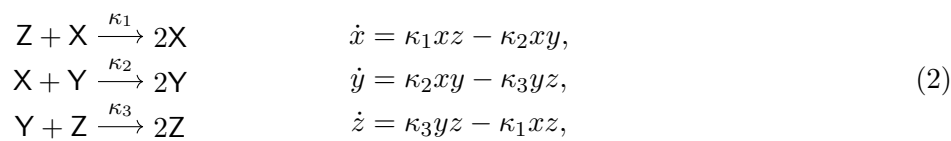
\begin{align}\label{eq:ivanova_only}
\begin{aligned}
    \begin{tikzpicture}[scale=0.6]

\node[left]  (P1) at (0, 0) {$\mathsf{Z}+\mathsf{X}$};
\node[right] (P2) at (1, 0) {$2\mathsf{X}$};
\node[left]  (P3) at (0,-1) {$\mathsf{X+Y}$};
\node[right] (P4) at (1,-1) {$2\mathsf{Y}$};
\node[left]  (P5) at (0,-2) {$\mathsf{Y+Z}$};
\node[right] (P6) at (1,-2) {$2\mathsf{Z}$};

\draw[->] (P1) to node[above] {\footnotesize $\kappa_1$} (P2);
\draw[->] (P3) to node[above] {\footnotesize $\kappa_2$} (P4);
\draw[->] (P5) to node[above] {\footnotesize $\kappa_3$} (P6);

\node at (8,-1)
{$\begin{aligned}
\dot{x} &= \kappa_1 x z - \kappa_2 x y, \\
\dot{y} &= \kappa_2 x y - \kappa_3 y z, \\
\dot{z} &= \kappa_3 y z - \kappa_1 x z,
\end{aligned}$};

\end{tikzpicture}
\end{aligned}
\end{align}
respectively. Both networks are \emph{bimolecular}, i.e., the molecularity of every source and target complex is at most two. Both have rank two, i.e., the span of their reaction vectors is two-dimensional; for the Ivanova reactions this follows from the mass-conservation relation $\dot{x}+\dot{y}+\dot{z}=0$. In both systems, all positive nonequilibrium solutions are periodic and, up to the inclusion of trivial species (to be defined below), these are the only three-reaction, bimolecular, rank-two, mass-action systems which admit a periodic solution. In fact, even with any number of reactions, there are no bimolecular rank-two systems that admit isolated periodic orbits \cite{pota:1983}, \cite{pota:1985}, \cite[Theorem 4.1]{boros:hofbauer:2022a}.

Hence, when searching for small mass-action systems with isolated periodic orbits, it is natural to study bimolecular rank-three networks or trimolecular rank-two networks. The former were studied in, for example, \cite{wilhelm:heinrich:1995} and \cite{boros:hofbauer:2023}, and more systematically in \cite{banaji:2018} and \cite{banaji:boros:2023}. Examples of the latter include Selkov's glycolytic oscillator \cite{selkov:1968}, the Brusselator \cite{lefever:nicolis:1971}, and the Schnakenberg networks \cite{schnakenberg:1979}.

In the present paper, we study networks with bimolecular sources but allow higher target molecularity. For instances of this kind in the literature, see, for example, \cite{frank-kamenetsky:salnikov:1943} or \cite{farkas:noszticzius:1985}. Bimolecular sources are chemically more realistic than sources of higher molecularity, and also easier to treat mathematically, because the corresponding mass-action differential equation is only quadratic. From here onwards, for brevity, we refer to networks with bimolecular sources as \lq\lq quadratic\rq\rq. For example, a trimolecular, quadratic network will mean a network with source molecularities at most two and target molecularities at most three. 

As our goal is to find or rule out oscillation in small networks, we focus on networks with three reactions, the minimum that is necessary for oscillation. The following theorem is one of our main results. It is an immediate corollary of Theorem~\ref{thm:ndim_trimolec}, which is proved in \Cref{sec:n_3_2}.

\begin{theorem}\label{thm:intro_trimolec}
Three-reaction, trimolecular, quadratic, mass-action systems admit no isolated periodic orbit.
\end{theorem}
\noindent We remark that there is no assumption in \Cref{thm:intro_trimolec} on the number of species involved: regardless of the number of species, whenever a three-reaction, trimolecular, quadratic, mass-action system has a periodic orbit, all nearby orbits are also periodic. In fact, we show that systems in this class admitting periodic orbits must belong to one of three families: one related to the Lotka system \eqref{eq:lotka_only}, one to the Ivanova system \eqref{eq:ivanova_only}, and one to the \emph{Lifted LVA}
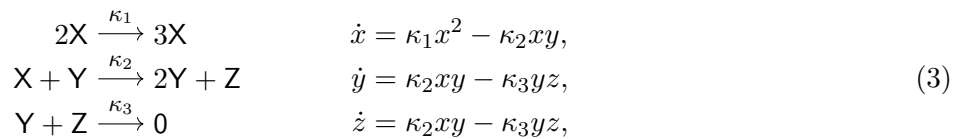
\begin{align}\label{eq:LVA_lifted}
\begin{aligned}
\begin{tikzpicture}[scale=0.6]

\node[left]  (P1) at (0, 0) {$2\mathsf{X}$};
\node[right] (P2) at (1, 0) {$3\mathsf{X}$};
\node[left]  (P3) at (0,-1) {$\mathsf{X+Y}$};
\node[right] (P4) at (1,-1) {$2\mathsf{Y}+\mathsf{Z}$};
\node[left]  (P5) at (0,-2) {$\mathsf{Y+Z}$};
\node[right] (P6) at (1,-2) {$\mathsf{0}$};

\draw[->] (P1) to node[above] {\footnotesize $\kappa_1$} (P2);
\draw[->] (P3) to node[above] {\footnotesize $\kappa_2$} (P4);
\draw[->] (P5) to node[above] {\footnotesize $\kappa_3$} (P6);

\node at (8,-1)
{$\begin{aligned}
\dot{x} &= \kappa_1 x^2 - \kappa_2 xy, \\
\dot{y} &= \kappa_2 xy - \kappa_3 yz, \\
\dot{z} &= \kappa_2 xy - \kappa_3 yz,
\end{aligned}$};

\end{tikzpicture}
\end{aligned}
\end{align}
a mass-action system that is obtained by adding a new species to the Lotka--Volterra--Autocatalator (LVA) \cite[Eq.\ (8)]{farkas:noszticzius:1985}, \cite[Eq.\ (1)]{simon:1992}, in such a way that the rank of the network remains two. The Lifted LVA admits a vertical Andronov--Hopf bifurcation: it has a two-parameter family of periodic orbits when $\kappa_2=\kappa_3 > \kappa_1$, and no periodic orbits otherwise. This is in contrast to the Lotka and Ivanova systems, where every positive nonequilibrium solution is periodic for \emph{all} $\kappa_1$, $\kappa_2$, $\kappa_3$, and no bifurcations are admitted.

In light of \Cref{thm:intro_trimolec}, a three-reaction, quadratic, mass-action system with an isolated periodic orbit must have a target complex with molecularity at least four. We find that there are four planar, three-reaction, tetramolecular, quadratic, mass-action systems which admit a supercritical Andronov--Hopf bifurcation, and thus a linearly stable limit cycle, the simplest one being
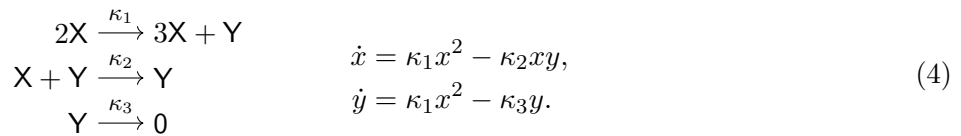
\begin{align}\label{eq:tetra_simplest}
\begin{aligned}
\begin{tikzpicture}[scale=0.6]

\node[left]  (P1) at (0, 0) {$2\mathsf{X}$};
\node[right] (P2) at (1, 0) {$3\mathsf{X}+\mathsf{Y}$};
\node[left]  (P3) at (0,-1) {$\mathsf{X+Y}$};
\node[right] (P4) at (1,-1) {$\mathsf{Y}$};
\node[left]  (P5) at (0,-2) {$\mathsf{Y}$};
\node[right] (P6) at (1,-2) {$\mathsf{0}$};

\draw[->] (P1) to node[above] {\footnotesize $\kappa_1$} (P2);
\draw[->] (P3) to node[above] {\footnotesize $\kappa_2$} (P4);
\draw[->] (P5) to node[above] {\footnotesize $\kappa_3$} (P6);

\node at (8,-1)
{$\begin{aligned}
\dot{x} &= \kappa_1 x^2 - \kappa_2 xy, \\
\dot{y} &= \kappa_1 x^2 - \kappa_3 y. \\
\end{aligned}$};

\end{tikzpicture}
\end{aligned}
\end{align}
The other three are obtained from \eqref{eq:tetra_simplest} by replacing the target complex of the second reaction by $2\mathsf{Y}$, $3\mathsf{Y}$, or $4\mathsf{Y}$. In fact, exactly two families of three-reaction, planar, quadratic, mass-action systems admit a supercritical Andronov--Hopf bifurcation. The first family (with source complexes $2\mathsf{X}$, $\mathsf{X} +\mathsf{Y}$, $\mathsf{Y}$) includes the tetramolecular examples mentioned above.  In the second family (with source complexes $2\mathsf{X}$, $\mathsf{X} +\mathsf{Y}$, $\mathsf{0}$), every network has a target complex with a molecularity of at least seven. These results, along with further results on planar, quadratic, mass-action systems admitting periodic orbits, are discussed in \Cref{sec:2_3_2}.

The rest of this paper is organised as follows. After introducing the basic notation and terminology in \Cref{sec:prelim}, we present some tools focussed on the analysis of rank-two mass-action systems in \Cref{sec:rank_two}. Oscillations in two-species, three-reaction systems are studied in \Cref{sec:2_3_2}, and three-reaction networks with an arbitrary number of species are treated in \Cref{sec:n_3_2}. Finally, we close with some concluding remarks and observations in \Cref{sec:conclusions}.
\section{Preliminaries} \label{sec:prelim}

We collect some basic notation, terminology, and known results needed later.

The symbols $\mathbb{R}_+$, $\mathbb{R}_{\geq0}$, and $\mathbb{R}_-$ denote the set of \emph{positive}, \emph{nonnegative}, and \emph{negative real numbers}, respectively. Accordingly, $\mathbb{R}^n_+$, $\mathbb{R}^n_{\geq0}$, and $\mathbb{R}^n_-$ denote the \emph{positive}, \emph{nonnegative}, and \emph{negative orthants}, respectively. We use similar notation for sets of positive or nonnegative integers. We refer to subsets of $\mathbb{R}^n_+$ as {\em positive}.

Given a row vector $a=[a_1,\ldots,a_n]$ of nonnegative integers, we adopt the standard convention that $x^a$ is an abbreviation for the monomial $x_1^{a_1}\cdots x_n^{a_n}$. Accordingly, if $A$ is an $m \times n$ matrix of nonnegative integers with $a_{j\cdot}$ being its $j$th row ($j=1,\ldots,m$) then $x^A$ denotes the column vector $[x^{a_{1\cdot}},\ldots,x^{a_{m\cdot}}]^\top$.

The symbol $\circ$ stands for the \emph{entrywise product} of two vectors or matrices of the same size.

For $u,v \in \mathbb{R}^n$, we write $u \cdot v$ for the \emph{scalar product} of $u$ and $v$. When $n=3$, we denote the \emph{cross product} of $u$ and $v$ by $u \times v$.

\subsection{Chemical reaction networks}

We start by introducing (chemical) species, (chemical) complexes, (chemical) reactions, and (chemical reaction) networks. For a more detailed exposition, the reader may consult, for example, \cite{yu:craciun:2018}.

Given \emph{species} $\mathsf{X}_1,\ldots,\mathsf{X}_n$, a \emph{complex} is a formal sum $\sum_{i=1}^na_i\mathsf{X}_i$, where the coefficients $a_i$ are assumed to be nonnegative integers. A \emph{reaction} corresponds to the conversion of a complex termed the \emph{source complex} (or just \emph{source} for short) into another termed the \emph{target complex} (or just \emph{target} for short): a reaction can thus be regarded as an ordered pair of complexes. A \emph{network} is a collection of reactions.

To facilitate the introduction of further terminology, consider the $m$-reaction network
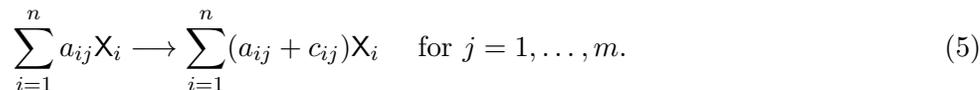
\begin{align}\label{eq:genCRN}
\sum_{i=1}^na_{ij}\mathsf{X}_i \longrightarrow \sum_{i=1}^n(a_{ij}+c_{ij})\mathsf{X}_i \quad \text{ for } j = 1, \ldots, m.
\end{align}
The matrix $\Gamma = [c_{ij}]\in\mathbb{R}^{n \times m}$ is called the \emph{stoichiometric matrix} of the network, while $\Gamma_l = [a_{ij}]\in\mathbb{R}^{n \times m}$ is termed its \emph{source matrix}. Each column of $\Gamma$ is the \emph{reaction vector} of the corresponding reaction. The image of $\Gamma$, denoted by $\im\Gamma$, is termed the \emph{stoichiometric subspace} of the network, while the sets $(x_0 + \im\Gamma)\cap\mathbb{R}^n_{\geq0}$ for $x_0 \in \mathbb{R}^n_{\geq0}$, and $(x_0 + \im\Gamma)\cap\mathbb{R}^n_+$ for $x_0 \in \mathbb{R}^n_+$, are the network's \emph{stoichiometric classes} and \emph{positive stoichiometric classes}, respectively. Finally, the \emph{rank} of the network is, by definition, $\rank \Gamma$.

A species $\mathsf{X}_i$ is called \emph{trivial} if $c_{ij}=0$ for all $j=1,\ldots,m$. In all reasonable models of a reaction network, the concentration of a trivial species remains constant along every trajectory.

A reaction network can be identified with its {\em Euclidean embedded graph} as defined in \cite{craciun:2019}. This is a directed graph obtained by identifying each complex with a point in $\mathbb{Z}^n_{\geq 0}$, and each reaction with an arc whose tail is the source of the reaction and whose head is the target of the reaction. For example, the Euclidean embedded graph of the tetramolecular network admitting a supercritical Andronov--Hopf bifurcation in \eqref{eq:tetra_simplest} is:

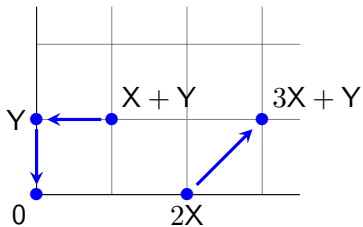
\begin{align*}
    \begin{tikzpicture}

\draw [step=1, gray, very thin] (0,0) grid (3.5,2.5);
\draw [ -, black] (0,0)--(3.5,0);
\draw [ -, black] (0,0)--(0,2.5);

\node[inner sep=0,outer sep=1] (P1) at (2,0) {\large \textcolor{blue}{$\bullet$}};
\node[inner sep=0,outer sep=1] (P2) at (3,1) {\large \textcolor{blue}{$\bullet$}};
\node[inner sep=0,outer sep=1] (P3) at (1,1) {\large \textcolor{blue}{$\bullet$}};
\node[inner sep=0,outer sep=1] (P4) at (0,1) {\large \textcolor{blue}{$\bullet$}};
\node[inner sep=0,outer sep=1] (P5) at (0,0) {\large \textcolor{blue}{$\bullet$}};

\node [below]  at (P1) {$2\mathsf{X}$};
\node [above right] at (P2) {$3\mathsf{X}+\mathsf{Y}$};
\node [above right] at (P3) {$\mathsf{X}+\mathsf{Y}$};
\node [left]  at (P4) {$\mathsf{Y}$};
\node [below left]  at (P5) {$\mathsf{0}$};

\draw[arrows={-stealth},very thick,blue] (P1) to node {} (P2);
\draw[arrows={-stealth},very thick,blue] (P3) to node {} (P4);
\draw[arrows={-stealth},very thick,blue] (P4) to node {} (P5);

\end{tikzpicture}
\end{align*}

\subsection{Molecularity}

The \emph{molecularity} of a complex $\sum_{i=1}^na_{i}\mathsf{X}_i$ is the nonnegative integer $\sum_{i=1}^na_{i}$. If every complex of a network has molecularity at most two, three, four, etc., then the network is said to be \emph{bimolecular}, \emph{trimolecular}, \emph{tetramolecular}, etc. For example, the network \eqref{eq:genCRN} is bimolecular if and only if $\sum_{i=1}^na_{ij} \leq 2$ and $\sum_{i=1}^n(a_{ij}+c_{ij}) \leq 2$ for $j=1, \ldots, m$.

If every source complex of a network has molecularity at most two then we refer to it as a \emph{quadratic network}. If, for example, the network \eqref{eq:genCRN} satisfies $\sum_{i=1}^na_{ij} \leq 2$ and $\sum_{i=1}^n(a_{ij}+c_{ij}) \leq 3$ for $j=1, \ldots, m$, we refer to it as a \lq\lq trimolecular, quadratic network\rq\rq. In terms of the Euclidean embedded graph, such a network is one whose sources are confined to $\{a \in \mathbb{Z}^n_{\geq 0}\colon \sum_{i=1}^n a_i \leq 2\}$, and whose targets are confined to $\{a \in \mathbb{Z}^n_{\geq 0}\colon\sum_{i=1}^n a_i \leq 3\}$.

\subsection{Dynamically nontrivial networks}\label{subsec:dyn_nontriv}

We refer to a network as \emph{dynamically nontrivial} if $\ker \Gamma \cap \mathbb{R}^m_+ \neq \emptyset$ or equivalently, $\im \Gamma^\top \cap \mathbb{R}^m_{\geq0}=\{0\}$, and \emph{dynamically trivial} otherwise. The equivalence of the two definitions follows from Stiemke's Theorem \cite{stiemke:1915}, which is a variant of Farkas' Lemma. Under weak assumptions on the reaction rates, the existence of a nonzero, nonnegative vector in the image of $\Gamma^\top$ is equivalent to the existence of a linear Lyapunov function in $\mathbb{R}^n_+$ for the associated differential equation, which increases strictly along all orbits in $\mathbb{R}^n_+$ (see \cite[Section 3.3]{banaji:2017}). Thus, dynamically trivial networks do not admit positive limit sets. In particular, only dynamically nontrivial networks can have an equilibrium or a periodic orbit in $\mathbb{R}^n_+$. We remark that some authors refer to dynamically nontrivial networks as \emph{consistent} \cite{angeli:leenheer:sontag:2007}, and similar ideas appear already in \cite[Section 5]{feinberg:1987}. 

Notice that we can interpret the condition for a network to be dynamically nontrivial as saying that its reaction vectors must be {\em positively dependent}: the zero vector in $\mathbb{R}^n$ can be written as a positive combination of the $m$ reaction vectors. In particular, a dynamically nontrivial network with $m$ reactions has rank at most $m-1$. Moreover, as we will see in the next subsection, a mass-action system with $m$ reactions and rank $m$ admits no periodic orbits, positive or otherwise. Since in this paper we are interested in networks with three reactions with the potential for periodic orbits, the networks of interest have rank at most two. On the other hand, since the differential equations we investigate are autonomous and have a unique solution for each initial condition in $\mathbb{R}^n_{\geq0}$, periodic solutions can only occur for networks of rank at least two. Thus, our main focus is on rank-two networks. In \Cref{sec:rank_two} we discuss some properties of rank-two mass-action systems.

\subsection{Mass-action systems}
\label{subsec:MA}

Assuming \emph{mass-action kinetics}, a positive number, termed the \emph{rate constant}, is associated with each reaction. The species \emph{concentration} $x \in \mathbb{R}^n_{\geq0}$ then evolves over time according to the autonomous ordinary differential equation
\begin{align}\label{eq:mass_action_ode}
\dot{x} = \Gamma (\kappa \circ x^{\Gamma_l^\top}),
\end{align}
where $\kappa\in\mathbb{R}^m_+$ is the vector of the rate constants. By a \emph{mass-action system} we mean a network with rate constants, or the differential equation \eqref{eq:mass_action_ode} itself; this should cause no confusion.

It can be shown that both the positive orthant $\mathbb{R}^n_+$ and the nonnegative orthant $\mathbb{R}^n_{\geq0}$ are forward invariant under \eqref{eq:mass_action_ode}. In fact, solutions with a positive initial condition are confined to the positive stoichiometric class of the initial condition for all $t \geq 0$. It is also well-known that given any (relatively open) face $F$ of $\mathbb{R}^n_{\geq 0}$, the mass-action vector field on $F$  is either nowhere tangent to $F$, or everywhere tangent to $F$, in which case $F$ is locally invariant and forward invariant. If we restrict attention to any such locally invariant face then, by removing species whose concentrations are zero on the face and reactions involving these species, we obtain either a mass-action system involving fewer species; or an \lq\lq empty\rq\rq system where no reactions proceed and hence $F$ consists entirely of equilibria. 

The previous construction sometimes allows us to extend claims about the positive orthant to the nonnegative orthant as a whole. For example, we observe that a mass-action system with $m$ reactions and rank $m$ (i.e., with linearly independent reaction vectors) necessarily forbids periodic orbits. That positive periodic orbits are forbidden is immediate as the network is dynamically trivial. However, periodic orbits are also forbidden on any locally invariant face of $\mathbb{R}^n_{\geq 0}$: restricting attention to such a face we obtain either a system which again has linearly independent reaction vectors and hence is dynamically trivial; or one where no reactions proceed and all points are equilibria. We can infer that a three-reaction mass-action system with a periodic orbit must have rank two.

\subsection{The reduced Jacobian determinant of a mass-action system}

We will refer to a network with $n$ species, $m$ reactions, and rank $r$ as an $(n,m,r)$ network. The Jacobian matrix of an $(n,m,r)$ mass-action system is, at each point of $\mathbb{R}^n_{\geq 0}$, an $n \times n$ matrix of rank at most $r$. We are interested in the dynamics of such a system restricted to stoichiometric classes, and hence in the action of its Jacobian matrices on the stoichiometric subspace. Fixing any $x_0 \in \mathbb{R}^n_{\geq 0}$, the \emph{reduced Jacobian determinant} of the system at $x_0$ is the determinant of the Jacobian matrix at $x_0$ regarded as a linear transformation on the stoichiometric subspace. Given any basis for the stoichiometric subspace, we can write down a matrix representation of this linear transformation. Since all such matrices are similar, we abuse notation by referring to any one of them as the \emph{reduced Jacobian matrix} of the system at $x_0$. The reduced Jacobian determinant is then just the determinant of any reduced Jacobian matrix. Equivalently, it is the product of the $r$ eigenvalues of the Jacobian matrix associated with the stoichiometric subspace. A number of equivalent formulations are given in a more general setting in \cite[Section 2.2]{banaji:pantea:2016}. We refer to an equilibrium of a mass-action system as {\em nondegenerate} if the reduced Jacobian determinant, evaluated at the equilibrium, is nonzero.

\subsection{Nondegenerate \texorpdfstring{$(n,n+1,n)$}{} networks}

The class of $(n,n+1,n)$ networks is analysed in \cite[Section 3]{banaji:boros:2023}. A dynamically nontrivial $(n,n+1,n)$ network is called \emph{nondegenerate} if its source complexes are affinely independent, and \emph{degenerate} otherwise. The terminology is justified by the following result (see \cite[Lemma 3.1 and Remark 3.2]{banaji:boros:2023}).

\begin{lemma}\label{lem:n_n+1_n_equil}
Consider a dynamically nontrivial $(n,n+1,n)$ network. Then the following statements hold.
\begin{enumerate}[label = {(\alph*)}]
\item If the network is nondegenerate then the associated mass-action system has a unique positive equilibrium for all choices of rate constants, and this equilibrium is nondegenerate. 
\item If the network is degenerate then the associated mass-action system has no isolated positive equilibria.
\end{enumerate}
\end{lemma}
We remark that as a consequence of \Cref{lem:sources_on_a_line} below, affine independence of the sources is a necessary condition for oscillation in any $(n,3,2)$ mass-action system. Thus nondegeneracy of the network is necessary for oscillation in $(2,3,2)$ networks.

\section{Rank-two mass-action systems} \label{sec:rank_two}

In this section, we derive a few properties of mass-action systems whose underlying network has rank two, focussing on necessary conditions for periodic orbits, or isolated periodic orbits.

\subsection{Periodic orbits in rank-two mass-action systems}\label{subsec:saddle_no_periodic}

We say that a mass-action system ``admits'' a periodic orbit if it has a periodic orbit for some choice of rate constants. Observe that a periodic orbit in a rank-two mass-action system without any trivial species must necessarily be positive (and, consequently, the system must be dynamically nontrivial). This follows from the simple observation that the intersection of any stoichiometric class with a proper face $F$ of $\mathbb{R}^n_{\geq 0}$ must have dimension less than two unless all species whose concentrations vanish everywhere on $F$ are trivial. 


We refer to a positive equilibrium of a rank-two mass-action system as a \emph{saddle} if the reduced Jacobian determinant, evaluated at the equilibrium, is negative. Equivalently, one of the nontrivial eigenvalues associated with the equilibrium is positive, while the other is negative. A periodic orbit in a rank-two system must contain at least one equilibrium in its interior on the stoichiometric class on which it resides, and not all of these equilibria can be saddles (see e.g.\ \cite[Section 3.5]{farkas:1994}). Thus, positive periodic orbits are ruled out in a rank-two mass-action system where all positive equilibria are saddles. If, additionally, the network has no trivial species, then (by our observations in the previous paragraph) periodic orbits on the boundary of $\mathbb{R}^n_{\geq 0}$ are ruled out too. We will frequently use these observations to rule out periodic orbits in mass-action systems.

\subsection{Sources on a line}

We make an observation about rank-two mass-action systems whose source complexes lie on a line. There is no assumption about the number of species or reactions, or about the molecularities of the complexes.

\begin{lemma}\label{lem:sources_on_a_line}
A rank-two mass-action system whose source complexes lie on a line admits no periodic orbit.
\end{lemma}
\begin{proof}
Since the positive orthant $\mathbb{R}^n_+$ is forward invariant, a periodic orbit lies either entirely in $\mathbb{R}^n_+$ or on the boundary of $\mathbb{R}^n_+$. In the latter case, the periodic orbit must lie entirely in some relatively open proper face of $\mathbb{R}^n_{\geq 0}$ (see \Cref{subsec:MA}).

We first show that the system admits no periodic orbit in $\mathbb{R}^n_+$. The general form of the mass-action differential equation for a network with $n$ species and $m$ reactions is
\begin{align*}
\dot{x}_i = \sum_{j=1}^m \beta_{ij} x^{\alpha_j} \quad (i = 1, \ldots, n),
\end{align*}
where $\alpha_j \in \mathbb{R}^n$ represents the $j$th source complex ($j=1,\ldots,m$). Since all source complexes lie on a line, there exist $s_1, \ldots, s_m\in \mathbb{R}$ and $\alpha \in \mathbb{R}^n$ such that $\alpha_j = \alpha_1 + s_j \alpha$ (here $s_1=0$). Thus, after division by the positive scalar function $x^{\alpha_1}$, we are left with
\begin{align*}
\dot{x}_i = f_i(x^\alpha) \quad (i = 1, \ldots, n),
\end{align*}
a differential equation whose r.h.s.\ depends on $x\in\mathbb{R}^n_+$ only through the scalar $x^\alpha$. Since the rank of the network is two, the positive orthant is foliated by two-dimensional invariant linear manifolds, the positive stoichiometric classes. By \cite[Proposition 1]{nitecki:1978}, none of these contains a periodic orbit.

Finally, we argue that a periodic orbit on the boundary of $\mathbb{R}^n_{\geq0}$ is not possible either. Suppose, by way of contradiction, that some proper $k$-dimensional (relatively open) face $F$ of $\mathbb{R}^n_{\geq 0}$ includes a periodic orbit. Clearly, $k \geq 2$ and, by remarks in \Cref{subsec:MA}, the face $F$ must be locally invariant. Restricting attention to $F$ (which we identify with $\mathbb{R}^k_+$), and removing species whose concentrations vanish on $F$ and reactions involving these species, we can regard the system as a mass-action system on $k$ species, of rank at most two, and with a periodic orbit which now lies in $\mathbb{R}^k_+$. But this is ruled out by the arguments in \Cref{subsec:MA} and in the previous paragraph. This concludes the proof.
\end{proof}

\subsection{The reduced Jacobian determinant of a three-reaction system} \label{subsec:detJred}

Three-reaction, rank-two, mass-action systems are the class of systems of main interest in this paper. In this subsection, we derive a formula for the reduced Jacobian determinant of such a system at any positive equilibrium, and demonstrate how this formula can be applied to rule out oscillation. 

Since positive equilibria are ruled out for dynamically trivial networks, we consider a dynamically nontrivial $(n,3,2)$ mass-action system, with no assumptions on the molecularity of complexes. W.l.o.g.\ we may assume that the first and the second rows of the stoichiometric matrix $\Gamma\in\mathbb{R}^{n \times 3}$, say $[c_1,c_2,c_3]$ and $[d_1,d_2,d_3]$, form a basis for its row-space, and write
\begin{align*}
\Gamma = \widetilde{\Gamma}\begin{bmatrix} c_1 & c_2 & c_3 \\ d_1 & d_2 & d_3 \end{bmatrix},
\end{align*}
where $\widetilde{\Gamma}\in\mathbb{R}^{n\times 2}$ has rank two, and its top $2 \times 2$ block is the identity matrix. Defining $c=[c_1,c_2,c_3]^\top$ and $d=[d_1,d_2,d_3]^\top$, the kernel of $\Gamma$ is spanned by $u=c \times d$ (i.e., $u_1 = c_2 d_3 - c_3 d_2$, $u_2 = c_3 d_1 - c_1 d_3$, $u_3 = c_1 d_2 - c_2 d_1$). Since the network is dynamically nontrivial, $u \in \mathbb{R}^3_+ \cup \mathbb{R}^3_-$. Fix $\kappa \in \mathbb{R}^3_+$ and let $\bar{x} \in \mathbb{R}^n_+$ be an equilibrium, i.e., $\kappa \circ \bar{x}^{\Gamma_l^\top} = \mu u$ for some nonzero scalar $\mu$. We may use $\widetilde{\Gamma}$ to define local coordinates on the stoichiometric class of $\bar{x}$ in the natural way and obtain the reduced Jacobian matrix at $\bar{x}$:
\begin{align*}
J_{\mathrm{red}} = \mu \begin{bmatrix} c^\top \\ d^\top \end{bmatrix} \Delta_u \Gamma_l^\top \Delta_{1/\bar{x}}\widetilde{\Gamma},
\end{align*}
where $\Delta_u \in \mathbb{R}^{3\times 3}$ and $\Delta_{1/\bar{x}} \in \mathbb{R}^{n\times n}$ are the diagonal matrices with $u_j$ ($j=1,2,3$) and $1/\bar{x}_i$ ($i=1,\ldots,n$) on their diagonal, respectively (for more details, see \cite[Appendix A]{banaji:pantea:2016} and \cite[Sections 2 and 3]{banaji:boros:2023}). Writing $a_{i\cdot}$ for the $i$th row of $\Gamma_l$, application of the Cauchy--Binet formula leads to
\begin{align}\label{eq:detJred}
\det J_{\mathrm{red}} = \mu |\mu| |u_1 u_2 u_3|\sum_{i<j}\frac{\widetilde{\Gamma}[\{i,j\},\{1,2\}]}{\bar{x}_i \bar{x}_j}(\bm{1}\cdot (a_{i\cdot} \times a_{j\cdot})),
\end{align}
where $\widetilde{\Gamma}[\{i,j\},\{1,2\}]$ is the determinant of the $2 \times 2$ submatrix of $\widetilde{\Gamma}$ that is formed of its $i$th and $j$th row, $\bm{1}$ is the vector of ones in $\mathbb{R}^3$, and we also used the fact that $\mu u_k = |\mu u_k|$. Note that the sign of $\bm{1}\cdot (a_{i\cdot} \times a_{j\cdot})$ tells us about the orientation of the three source complexes projected onto the $(i,j)$th coordinate. We remark that any choice of privileged species must lead to the same value of $\det J_{\mathrm{red}}$.

In the special case $n=2$, the only positive stoichiometric class is $\mathbb{R}^2_+$ itself, $\widetilde{\Gamma}$ is the identity matrix, and the reduced Jacobian determinant is just the Jacobian determinant. Hence,
\begin{align}\label{eq:detJ_planar}
    \det J = \frac{\mu |\mu| |u_1 u_2 u_3|}{\bar{x}_1 \bar{x}_2}(\bm{1}\cdot (a_{1\cdot} \times a_{2\cdot})).
\end{align}
From this expression we observe that for $J$ to have a nonzero determinant, the three source complexes must be affinely independent (i.e., they must span a triangle). Furthermore, since $\bm{1} \cdot (c \times d)$ has the same sign as $\mu$, $\det J > 0$ is positive if and only if the three reaction vectors $[c_1,d_1]^\top$, $[c_2,d_2]^\top$, $[c_3,d_3]^\top$ span a triangle with the same orientation as the triangle spanned by the three source complexes. If the orientations of the triangles spanned by the sources and the reaction vectors are opposite to each other, then $\det J<0$, and the equilibrium is a saddle. We illustrate the two possibilities with the following two networks: on the left is the LVA (see \eqref{eq:LVA_basic} below); and on the right is a network with the same complexes as the LVA, but different reaction vectors. Any positive equilibrium of the latter network is necessarily a saddle.
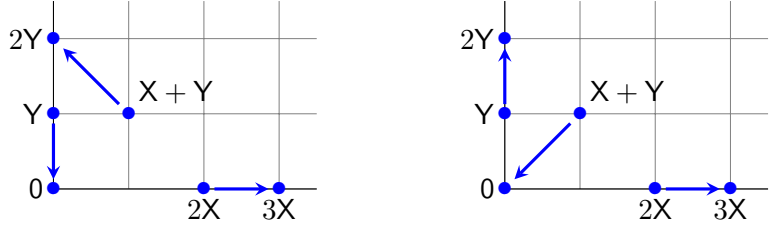
\begin{align*}
    \begin{tikzpicture}

\draw [step=1, gray, very thin] (0,0) grid (3.5,2.5);
\draw [ -, black] (0,0)--(3.5,0);
\draw [ -, black] (0,0)--(0,2.5);

\node[inner sep=0,outer sep=1] (P1) at (2,0) {\large \textcolor{blue}{$\bullet$}};
\node[inner sep=0,outer sep=1] (P2) at (3,0) {\large \textcolor{blue}{$\bullet$}};
\node[inner sep=0,outer sep=1] (P3) at (1,1) {\large \textcolor{blue}{$\bullet$}};
\node[inner sep=0,outer sep=1] (P4) at (0,2) {\large \textcolor{blue}{$\bullet$}};
\node[inner sep=0,outer sep=1] (P5) at (0,1) {\large \textcolor{blue}{$\bullet$}};
\node[inner sep=0,outer sep=1] (P6) at (0,0) {\large \textcolor{blue}{$\bullet$}};

\node [below]  at (P1) {$2\mathsf{X}$};
\node [below] at (P2) {$3\mathsf{X}$};
\node [above right] at (P3) {$\mathsf{X}+\mathsf{Y}$};
\node [left]  at (P4) {$2\mathsf{Y}$};
\node [left]  at (P5) {$\mathsf{Y}$};
\node [left]  at (P6) {$\mathsf{0}$};

\draw[arrows={-stealth},very thick,blue] (P1) to node {} (P2);
\draw[arrows={-stealth},very thick,blue] (P3) to node {} (P4);
\draw[arrows={-stealth},very thick,blue] (P5) to node {} (P6);

\begin{scope}[shift={(6,0)}]

\draw [step=1, gray, very thin] (0,0) grid (3.5,2.5);
\draw [ -, black] (0,0)--(3.5,0);
\draw [ -, black] (0,0)--(0,2.5);

\node[inner sep=0,outer sep=1] (P1) at (2,0) {\large \textcolor{blue}{$\bullet$}};
\node[inner sep=0,outer sep=1] (P2) at (3,0) {\large \textcolor{blue}{$\bullet$}};
\node[inner sep=0,outer sep=1] (P3) at (1,1) {\large \textcolor{blue}{$\bullet$}};
\node[inner sep=0,outer sep=1] (P4) at (0,0) {\large \textcolor{blue}{$\bullet$}};
\node[inner sep=0,outer sep=1] (P5) at (0,1) {\large \textcolor{blue}{$\bullet$}};
\node[inner sep=0,outer sep=1] (P6) at (0,2) {\large \textcolor{blue}{$\bullet$}};

\node [below]  at (P1) {$2\mathsf{X}$};
\node [below] at (P2) {$3\mathsf{X}$};
\node [above right] at (P3) {$\mathsf{X}+\mathsf{Y}$};
\node [left]  at (P4) {$\mathsf{0}$};
\node [left]  at (P5) {$\mathsf{Y}$};
\node [left]  at (P6) {$2\mathsf{Y}$};

\draw[arrows={-stealth},very thick,blue] (P1) to node {} (P2);
\draw[arrows={-stealth},very thick,blue] (P3) to node {} (P4);
\draw[arrows={-stealth},very thick,blue] (P5) to node {} (P6);

\end{scope}

\end{tikzpicture}
\end{align*}

To illustrate formula \eqref{eq:detJred} in the case where $n\geq3$, consider the $(4,3,2)$ network
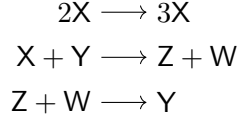
\begin{align*}
\begin{tikzpicture}[scale=0.6]

\node[left]  (P1) at (0, 0) {$2\mathsf{X}$};
\node[right] (P2) at (1, 0) {$3\mathsf{X}$};
\node[left]  (P3) at (0,-1) {$\mathsf{X+Y}$};
\node[right] (P4) at (1,-1) {$\mathsf{Z+W}$};
\node[left]  (P5) at (0,-2) {$\mathsf{Z+W}$};
\node[right] (P6) at (1,-2) {$\mathsf{Y}$};

\draw[->] (P1) to node[above] {} (P2);
\draw[->] (P3) to node[above] {} (P4);
\draw[->] (P5) to node[above] {} (P6);

\end{tikzpicture}
\end{align*}
which will play a role in \Cref{lem:4species_saddle} below. Here,
\begin{align*}
\Gamma = \begin{bmatrix*}[r]
    1 & -1 &  0 \\
    0 & -1 &  1 \\
    0 &  1 & -1 \\
    0 &  1 & -1
\end{bmatrix*} = 
\begin{bmatrix*}[r]
    1 &  0 \\
    0 &  1 \\
    0 & -1 \\
    0 & -1
\end{bmatrix*}
\begin{bmatrix*}
    1 &  -1 & 0 \\
    0 &  -1 & 1 
\end{bmatrix*}
\text{ and }
\Gamma_l = 
\begin{bmatrix*}
    2 & 1 & 0 \\
    0 & 1 & 0 \\
    0 & 0 & 1 \\
    0 & 0 & 1
\end{bmatrix*}.
\end{align*}
Therefore, with this choice of species ordering, $u_1=u_2=u_3=-1$ and $\mu<0$. Since the minors $\widetilde{\Gamma}[\{i,j\},\{1,2\}]$ vanish when $2 \leq i < j \leq 4$, only the choices $i=1$ and $j=2,3,4$ give nonzero terms in \eqref{eq:detJred}. We obtain
\begin{align*}
\det J_{\mathrm{red}} = \mu |\mu| \left(\frac{2}{\bar{x}\bar{y}} + \frac{1}{\bar{x}\bar{z}} + \frac{1}{\bar{x}\bar{w}}\right),
\end{align*}
which is negative since $\mu<0$. Hence, every positive equilibrium is a saddle within its stoichiometric class and, by the observations in \Cref{subsec:saddle_no_periodic}, the network admits no periodic orbits. 

\subsection{Necessary reactions for isolated periodic orbits in quadratic systems}

In this subsection, we use a divergence argument and the Bendixson--Dulac Test to show that the presence of certain reactions is necessary for the occurrence of isolated periodic orbits in rank-two, quadratic, mass-action systems. In the first result, we make no assumption on the number of reactions in the system. 

\begin{lemma}\label{lem:divergence}
Assume that a rank-two, quadratic network with no trivial species has no reaction of the form
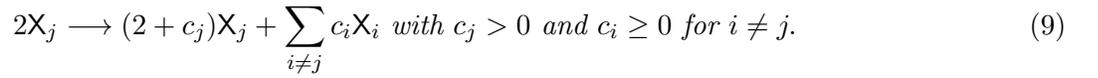
\begin{align}\label{eq:rxn_pos_div}
2\mathsf{X}_j \longrightarrow (2+c_j)\mathsf{X}_j + \sum_{i\neq j}c_i\mathsf{X}_i \text{ with } c_j>0 \text{ and } c_i\geq0 \text{ for }i\neq j.
\end{align}
Suppose that its associated mass-action system admits a periodic orbit. Then the network has either two or three species, and the mass-action differential equation reads either as
\begin{align}\label{eq:divergence_odes}
\begin{aligned}
\begin{tikzpicture}

\node at (0,0)
{$\begin{aligned}
\dot{x}_1 &= x_1(r_1+b_{12} x_2) \\
\dot{x}_2 &= x_2(r_2+b_{21} x_1)
\end{aligned}$};
\node at (0,-1.5) {with $r_1 r_2<0, r_1 b_{12}<0, r_2 b_{21}<0$};

\node at (4,0) {or};

\node at (8,0)
{$\begin{aligned}
    \dot{x}_1 &= x_1(b_{12} x_2 + b_{13} x_3) \\
    \dot{x}_2 &= x_2(b_{21} x_1 + b_{23} x_3) \\
    \dot{x}_3 &= x_3(b_{31} x_1 + b_{32} x_2)
\end{aligned}$};
\node at (8,-1.75) {$\begin{array}{c}\text{with } b_{12}b_{13} < 0, b_{21}b_{23} < 0, b_{31}b_{32} < 0 \\ \\[-3mm] \text{and } b_{12}b_{23}b_{31} + b_{13}b_{21}b_{32} = 0 \end{array}$};

\end{tikzpicture}
\end{aligned}
\end{align}
Consequently, a rank-two, quadratic, mass-action system with no trivial species and no reaction of the form \eqref{eq:rxn_pos_div} admits no isolated periodic orbits.
\end{lemma}
\begin{proof}
Consider a mass-action system satisfying the hypotheses of the lemma. By remarks in \Cref{subsec:saddle_no_periodic}, each periodic orbit of the system must be positive. Denote by $f$ the vector field on $\mathbb{R}^n_{+}$ obtained by multiplying the r.h.s.\ of the system by the Dulac function $(x_1\cdots x_n)^{-1}$. 
We consider the divergence $\mathrm{div} f = \sum_{j=1}^n\frac{\partial f_j}{\partial x_j}$. For a fixed $j\in\{1,\ldots,n\}$, a reaction $\sum_{i=1}^n a_i \mathsf{X}_i \stackrel{\kappa}{\longrightarrow} \sum_{i=1}^n (a_i+c_i) \mathsf{X}_i$ contributes the term $c_j\kappa x_1^{a_1-1}\cdots x_n^{a_n-1}$ to $f_j$. Hence, the contribution of such a reaction to $\mathrm{div} f$ is zero if $a_j=1$ or $c_j=0$, while it is negative if either $a_j=0$ and $c_j>0$ or $a_j=2$ and $c_j<0$. These are the only two possibilities, as reactions of the form \eqref{eq:rxn_pos_div} are excluded, and consequently, no reaction makes a positive contribution to $\mathrm{div} f$. Clearly, $\mathrm{div} f$ is either everywhere negative on $\mathbb{R}^n_+$, if there is some reaction which contributes a negative term, or is identically zero on $\mathbb{R}^n_{+}$ if there is no such reaction. 

If $\mathrm{div} f$ is negative everywhere in $\mathbb{R}^n_+$ then periodic orbits in $\mathbb{R}^n_+$ are precluded by the Bendixson--Dulac Test, see \cite[Theorem 3, Remark (3)]{li:1995}.

If, on the other hand, $\mathrm{div} f \equiv 0$ on $\mathbb{R}^n_+$, then the discussion above implies that for all $j$, $c_j\neq0$ implies $a_j=1$. Equivalently, the mass-action differential equation is a Lotka--Volterra equation with no diagonal term (i.e., the monomial $x_j^2$ does not occur in the expression for $\dot{x}_j$):
\begin{align*}
    \dot{x}_j = x_j\left(r_j + \sum_{k\neq j} b_{jk} x_k\right) \text{ for }j  = 1, \ldots, n.
\end{align*}

\noindent\textbf{Case $n=2$.} In order to admit a periodic orbit the system must admit a positive equilibrium which is not a saddle, from which it easily follows that $r_1 r_2<0, r_1 b_{12}<0, r_2 b_{21}<0$, as claimed.

\noindent\textbf{Case $n=3$.} Since the rank of the network is two, there is a nonzero $d \in \mathbb{R}^3$ such that $d_1 \dot{x}_1 + d_2 \dot{x}_2 + d_3 \dot{x}_3 = 0$. Hence, $d_1 r_1 = d_2 r_2 = d_3 r_3 = 0$ and 
\begin{align*}
\begin{bmatrix}
b_{12} & b_{21} & 0 \\
b_{13} & 0 & b_{31} \\
0 & b_{23} & b_{32}
\end{bmatrix}
\begin{bmatrix}
d_1 \\ d_2 \\ d_3 
\end{bmatrix}=
\begin{bmatrix}
0 \\ 0 \\ 0
\end{bmatrix}.
\end{align*}
This implies $b_{12}b_{23}b_{31} + b_{13}b_{21}b_{32} = 0$. Further, all $d_i$ must be nonzero, for otherwise either the system has a trivial species, or some species concentration increases strictly along positive trajectories, ruling out the existence of a periodic orbit (see \cite[proof of Theorem 4.1]{boros:hofbauer:2022a}). We conclude that $r_1=r_2=r_3=0$. The existence of a positive equilibrium then implies $b_{12}b_{13} < 0, b_{21}b_{23} < 0, b_{31}b_{32} < 0$.

\noindent\textbf{Case $n\geq4$.} One finds that, in fact, there is no rank-two network without trivial species whose mass-action differential equation is a Lotka--Volterra equation with no diagonal term. See \cite[proof of Theorem 4.1]{boros:hofbauer:2022a} for the details.

The nonexistence of isolated periodic orbits for the systems in \eqref{eq:divergence_odes} follows immediately from the fact that both vector fields have (nonlinear) first integrals on $\mathbb{R}^2_+$ and $\mathbb{R}^3_+$ respectively and, in fact, restricted to two-dimensional invariant sets, are Hamiltonian. The conserved quantities are $r_1\log y - r_2\log x +b_{12} y - b_{21}x$ (in the two-species case) and $d_1d_2b_{23}\log x+d_2d_3b_{31}\log y+d_1d_3b_{12}\log z$ (in the three-species case).
\end{proof}

Notice that the only trimolecular reaction of the form \eqref{eq:rxn_pos_div} is $2\mathsf{X}_j \longrightarrow 3 \mathsf{X}_j$. Hence, we obtain the following result, which is a major step towards proving one of our main results, namely, \Cref{thm:intro_trimolec}.
\begin{corollary}\label{cor:divergence_2X_3X}
If a rank-two, trimolecular, quadratic, mass-action network with no trivial species admits an isolated periodic orbit, then there exists $j\in\{1,\ldots,n\}$ such that $2\mathsf{X}_j \longrightarrow 3 \mathsf{X}_j$ is a reaction in the network.
\end{corollary}

If we restrict the scope of \Cref{lem:divergence} to the three-reaction case, we can characterize those networks that lead to the differential equations in \eqref{eq:divergence_odes}. The Lotka reactions \eqref{eq:lotka_only} and the Ivanova reactions \eqref{eq:ivanova_only} give such instances. However, in addition, we find a two-parameter family of two-species networks generalising the Lotka reactions, and a new two-parameter family of three-species networks, all of which which give rise to differential equations of the form in \eqref{eq:divergence_odes}. 

\begin{lemma}\label{lem:two_param_families}
Assume that a three-reaction, quadratic network with no trivial species has no reaction of the form \eqref{eq:rxn_pos_div}. Then its associated mass-action system admits a periodic orbit if and only if, up to a permutation of the species, the network is either the Ivanova reactions or belongs to one of the following two families with parameters $c,d\geq1$:
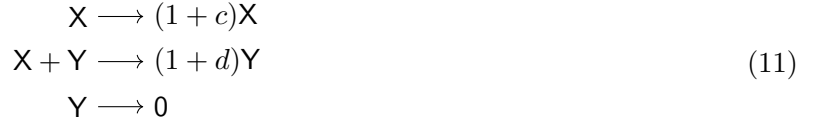
\begin{align}\label{eq:two_param_families_left}
\begin{aligned}
\begin{tikzpicture}[scale=0.6]

\node[left]  (P1) at (0, 0) {$\mathsf{X}$};
\node[right] (P2) at (1, 0) {$(1+c)\mathsf{X}$};
\node[left]  (P3) at (0,-1) {$\mathsf{X+Y}$};
\node[right] (P4) at (1,-1) {$(1+d)\mathsf{Y}$};
\node[left]  (P5) at (0,-2) {$\mathsf{Y}$};
\node[right] (P6) at (1,-2) {$\mathsf{0}$};

\draw[->] (P1) to node[above] {} (P2);
\draw[->] (P3) to node[above] {} (P4);
\draw[->] (P5) to node[above] {} (P6);

\end{tikzpicture}
\end{aligned}
\end{align}
and
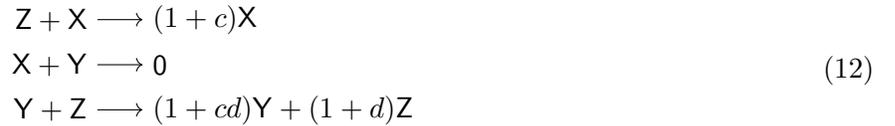
\begin{align}\label{eq:two_param_families_right}
\begin{aligned}
\begin{tikzpicture}[scale=0.6]

\node[left]  (P1) at (0, 0) {$\mathsf{Z}+\mathsf{X}$};
\node[right] (P2) at (1, 0) {$(1+c)\mathsf{X}$};
\node[left]  (P3) at (0,-1) {$\mathsf{X+Y}$};
\node[right] (P4) at (1,-1) {$\mathsf{0}$};
\node[left]  (P5) at (0,-2) {$\mathsf{Y+Z}$};
\node[right] (P6) at (1,-2) {$(1+cd)\mathsf{Y}+(1+d)\mathsf{Z}$};

\draw[->] (P1) to node[above] {} (P2);
\draw[->] (P3) to node[above] {} (P4);
\draw[->] (P5) to node[above] {} (P6);

\end{tikzpicture}
\end{aligned}
\end{align}
In the case of the Ivanova reactions and the networks in \eqref{eq:two_param_families_left}, for all rate constants, each stoichiometric class contains a unique positive equilibrium, which is a global center in its stoichiometric class. Depending on the rate constants, the networks in \eqref{eq:two_param_families_right} either have no periodic orbits; or they have, on some stoichiometric classes, a unique positive equilibrium which is a global center.
\end{lemma}
\begin{proof}
Consider a reaction network satisfying the hypotheses of the lemma. Recall from \Cref{subsec:MA} that the rank of a three-reaction system that admits a periodic orbit is necessarily two. By \Cref{lem:divergence}, the network has either two or three species, and the mass-action differential equation takes one of the forms in \eqref{eq:divergence_odes}. The two-species case is straightforward: the reader may confirm that any three-reaction network without trivial species giving rise to the differential equation on the left of \eqref{eq:divergence_odes} must be of the form in \eqref{eq:two_param_families_left}. Any such network admits a first integral $d\kappa_2 x+\kappa_2 y-\kappa_3\log x - c\kappa_1\log y$ on $\mathbb{R}^2_+$, which has bounded level sets and a global minimum at the unique positive equilibrium. It follows immediately that the equilibrium is a global center. We focus on the slightly harder three-species case.

It is easily seen that if a three-reaction, three-species network leads to the differential equation on the right in \eqref{eq:divergence_odes}, it is necessarily of the form
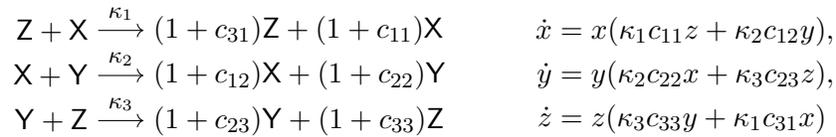
\begin{align*}
\begin{tikzpicture}[scale=0.6]

\node[left]  (P1) at (0, 0) {$\mathsf{Z}+\mathsf{X}$};
\node[right] (P2) at (1, 0) {$(1+c_{31})\mathsf{Z}+(1+c_{11})\mathsf{X}$};
\node[left]  (P3) at (0,-1) {$\mathsf{X+Y}$};
\node[right] (P4) at (1,-1) {$(1+c_{12})\mathsf{X}+(1+c_{22})\mathsf{Y}$};
\node[left]  (P5) at (0,-2) {$\mathsf{Y+Z}$};
\node[right] (P6) at (1,-2) {$(1+c_{23})\mathsf{Y}+(1+c_{33})\mathsf{Z}$};

\draw[->] (P1) to node[above] {\footnotesize $\kappa_1$} (P2);
\draw[->] (P3) to node[above] {\footnotesize $\kappa_2$} (P4);
\draw[->] (P5) to node[above] {\footnotesize $\kappa_3$} (P6);

\node at (13,-1)
{$\begin{aligned}
\dot{x} &= x(\kappa_1 c_{11} z + \kappa_2 c_{12} y), \\
\dot{y} &= y(\kappa_2 c_{22} x + \kappa_3 c_{23} z), \\
\dot{z} &= z(\kappa_3 c_{33} y + \kappa_1 c_{31} x)
\end{aligned}$};

\end{tikzpicture}
\end{align*}
with $c_{ij}\geq -1$ and $\sgn c_{ii} = - \sgn c_{i,i+1}\neq0$. Hence, three of the six $c_{ij}$'s are equal to $-1$, and the other three are positive integers. By a short calculation, the rank of the network is two if and only if $c_{11}c_{22}c_{33}=-c_{12}c_{23}c_{31}$. If none of the $c_{ii}$'s is negative then $c_{12}=c_{23}=c_{31}=-1$ and $c_{11}=c_{22}=c_{33}=1$, leading to the Ivanova reactions. If exactly one of the $c_{ii}$'s is negative then (because of the cyclic symmetry of the species) w.l.o.g.\ we may assume $c_{22}=-1$, and $c_{11} > 0$ and $c_{33} > 0$. Hence, $c_{12}=c_{31}=-1$ and because of the rank condition, $c_{23} = c_{11} c_{33}$, leading to the family of networks in \eqref{eq:two_param_families_right}. The cases when two or three of the $c_{ii}$'s are negative can be reduced to the cases already discussed by swapping two species.

The behaviour of the Ivanova mass-action system is widely known \cite[page 630]{volpert:hudjaev:1985}. Some straightforward calculations demonstrate that the networks in \eqref{eq:two_param_families_right} indeed give rise to a differential equation with centers in some stoichiometric classes for certain rate constants. The set of positive equilibria is the ray $t(\kappa_1\kappa_3 cd,\kappa_1 c,\kappa_2)$ for $t>0$, the positive stoichiometric classes are given by $\mathcal{P}_D=\{(x,y,z) \in \mathbb{R}^3_+\colon x-y+cz=D\}$ for $D \in \mathbb{R}$, and the system has a conserved quantity $d\kappa_3\log x - \kappa_1\log y + \kappa_2\log z$. The analysis of the system reveals the following.
\begin{itemize}
    \item When $\kappa_1>\kappa_2+\kappa_3d$, there is a positive equilibrium in $\mathcal{P}_D$ if and only if $D<0$. In this case, the equilibrium is unique, and it is a global center, since the conserved quantity has compact level sets on $\mathcal{P}_D$ when $D<0$. 
    \item When $\kappa_1=\kappa_2+\kappa_3d$, the whole ray of positive equilibria lies in $\mathcal{P}_0$. In fact, every ray in $\mathcal{P}_0$ through the origin is invariant, and hence, the system has no periodic orbit.
    \item When $\kappa_1<\kappa_2+\kappa_3d$, there is a positive equilibrium in $\mathcal{P}_D$ if and only if $D>0$. In this case, the equilibrium is unique, and it is a saddle. Hence, by the discussion in \Cref{subsec:saddle_no_periodic}, the system has no periodic orbit.
\end{itemize}
\end{proof}
\noindent We remark that the family of networks in \eqref{eq:two_param_families_left} is exactly \cite[Eq. (5)]{farkas:noszticzius:1985}.

\section{The analysis of quadratic \texorpdfstring{$(2,3,2)$}{} systems} \label{sec:2_3_2}

In this section, we are interested in two-species, three-reaction, quadratic, mass-action systems. Our first main result on these systems, \Cref{thm:planar_trimolec}, is that they admit no isolated periodic orbits when target molecularities do not exceed three. We will later, in \Cref{thm:ndim_trimolec}, generalise this result to mass-action systems with arbitrary numbers of species; however, the simpler planar case is of interest in itself, and its proof contributes towards the proof of the more general result.

\begin{theorem} \label{thm:planar_trimolec}
Three-reaction, two-species, trimolecular, quadratic mass-action systems admit no isolated periodic orbit.
\end{theorem}

The proof of \Cref{thm:planar_trimolec} is completed in \Cref{subsec:planar_trimolec}. In \Cref{subsec:planar_hopf}, we arrive at our second main result about planar systems, \Cref{thm:ten_src_triples}, where we find \emph{all} three-reaction, planar, quadratic mass-action systems (with arbitrary target molecularities) which admit an Andronov--Hopf bifurcation. It turns out that there are two families of networks that admit supercritical Andronov--Hopf bifurcation, while another family admits a vertical Andronov--Hopf bifurcation. Exactly four of these networks are tetramolecular, while all other networks with an Andronov--Hopf bifurcation have a target complex with a molecularity of at least five.

\subsection{Setting} \label{subsec:planar_setting}

For readability, we use a slightly different notation in the two-species case than in the general case, following \cite[Section 5]{boros:hofbauer:2021a}. Throughout this section, we are analysing the three-reaction mass-action system
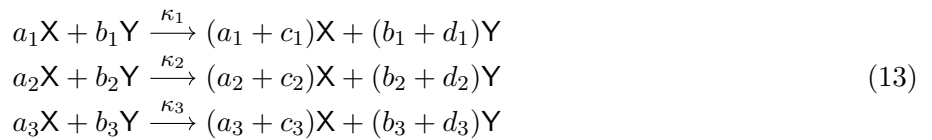
\begin{align} \label{eq:three_reactions_network}
\begin{aligned}
\begin{tikzpicture}[scale=0.6]

\node[left]  (P1) at (0, 0) {$a_1 \mathsf{X} + b_1 \mathsf{Y}$};
\node[right] (P2) at (1, 0) {$(a_1+c_1) \mathsf{X} + (b_1+d_1) \mathsf{Y}$};
\node[left]  (P3) at (0,-1) {$a_2 \mathsf{X} + b_2 \mathsf{Y}$};
\node[right] (P4) at (1,-1) {$(a_2+c_2) \mathsf{X} + (b_2+d_2) \mathsf{Y}$};
\node[left]  (P5) at (0,-2) {$a_3 \mathsf{X} + b_3 \mathsf{Y}$};
\node[right] (P6) at (1,-2) {$(a_3+c_3) \mathsf{X} + (b_3+d_3) \mathsf{Y}$};

\draw[->] (P1) to node[above] {\footnotesize $\kappa_1$} (P2);
\draw[->] (P3) to node[above] {\footnotesize $\kappa_2$} (P4);
\draw[->] (P5) to node[above] {\footnotesize $\kappa_3$} (P6);

\end{tikzpicture}
\end{aligned}
\end{align}
and its associated differential equation
\begin{align}
\label{eq:three_reactions_ode}
\begin{split}
\dot{x} &= c_1\kappa_1 x^{a_1}y^{b_1} + c_2\kappa_2 x^{a_2}y^{b_2} + c_3\kappa_3 x^{a_3}y^{b_3}, \\
\dot{y} &= d_1\kappa_1 x^{a_1}y^{b_1} + d_2\kappa_2 x^{a_2}y^{b_2} + d_3\kappa_3 x^{a_3}y^{b_3}
\end{split}
\end{align}
with $a_i$, $b_i$, $a_i+c_i$, $b_i+d_i$ ($i=1,2,3$) being nonnegative integers.

By \Cref{lem:sources_on_a_line}, equation \eqref{eq:three_reactions_ode} can have no periodic orbit if the three sources $(a_1,b_1)$, $(a_2,b_2)$, $(a_3,b_3)$ lie on a line. Hence, from here on, we assume that 
\begin{align}\label{eq:sources_triangle}
\text{$(a_1,b_1)$, $(a_2,b_2)$ and $(a_3,b_3)$ span a triangle, which is positively oriented.}
\end{align}
Clearly, the assumption on the orientation does not restrict generality.

The stoichiometric matrix of the network is 
\[
\Gamma= \begin{bmatrix}
    c_1 & c_2 & c_3 \\ d_1 & d_2 & d_3
\end{bmatrix}\,.
\]
Since we are interested in finding periodic orbits, we assume that $\rank \Gamma = 2$ and the network is dynamically nontrivial. Thus, with $c=[c_1,c_2,c_3]^\top$ and $d=[d_1,d_2,d_3]^\top$, the kernel of $\Gamma$ is spanned by $u=c \times d \in \mathbb{R}^3_+ \cup \mathbb{R}^3_-$. Here,
\begin{align}
\label{eq:u_formula}
\begin{split}
    u_1 &= c_2 d_3 - c_3 d_2, \\
    u_2 &= c_3 d_1 - c_1 d_3, \\
    u_3 &= c_1 d_2 - c_2 d_1.
\end{split}
\end{align}

By \Cref{lem:n_n+1_n_equil}, under the stated assumptions, the mass-action system \eqref{eq:three_reactions_ode} has a unique positive equilibrium $(\bar{x},\bar{y})$. As in \Cref{subsec:detJred}, there exists a nonzero $\mu\in \mathbb{R}$ such that $\mu u_i = \kappa_i \bar{x}^{a_i} \bar{y}^{b_i}$ ($i=1, 2, 3$), and the Jacobian matrix, denoted by $J$, at the equilibrium is given by
\begin{align*}
J = \mu\begin{bmatrix}
    \sum_{i=1}^3a_i c_i u_i & \sum_{i=1}^3b_i c_i u_i \\
    \sum_{i=1}^3a_i d_i u_i & \sum_{i=1}^3b_i d_i u_i
\end{bmatrix}
\begin{bmatrix}
    1/\bar{x} & 0 \\
    0 & 1/\bar{y}
\end{bmatrix}.
\end{align*}
By \Cref{lem:n_n+1_n_equil}, $\det J \neq 0$. In fact, since the source complexes are positively oriented by assumption, formula \eqref{eq:detJ_planar} implies that
\begin{align}\label{eq:sgn_chain}
    \sgn \det J = \sgn \mu = \sgn u_1 = \sgn u_2 = \sgn u_3.
\end{align}
When $\det J <0$, or equivalently $u \in \mathbb{R}^3_-$, the equilibrium is a saddle and, by the discussion in \Cref{subsec:saddle_no_periodic}, the mass-action system \eqref{eq:three_reactions_network} admits no periodic orbits.

For later use, the trace of the Jacobian matrix at the positive equilibrium is given by
\begin{align} \label{eq:trJ}
\tr J = \mu\left(\frac{1}{\bar{x}}\sum_{i=1}^3a_ic_iu_i + \frac{1}{\bar{y}}\sum_{i=1}^3b_id_iu_i\right).
\end{align}

\subsection{Lotka--Volterra--Autocatalator}\label{subsec:LVA}

In light of \Cref{cor:divergence_2X_3X}, a trimolecular network satisfying the assumptions of \Cref{thm:planar_trimolec} with no reaction $2\mathsf{X}_j \longrightarrow 3\mathsf{X}_j$ cannot have an isolated periodic orbit. This motivates the investigation of networks with $2\mathsf{X}_j \longrightarrow 3\mathsf{X}_j$.

Beginning with the Lotka reactions \eqref{eq:lotka_only} and replacing the linear autocatalytic step $\mathsf{X} \longrightarrow 2\mathsf{X}$ by the quadratic autocatalytic step $2\mathsf{X} \longrightarrow 3\mathsf{X}$, we obtain the system
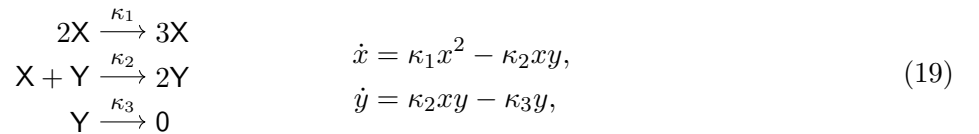
\begin{align}\label{eq:LVA_basic}
\begin{aligned}
    \begin{tikzpicture}[scale=0.6]

\node[left]  (P1) at (0, 0) {$2\mathsf{X}$};
\node[right] (P2) at (1, 0) {$3\mathsf{X}$};
\node[left]  (P3) at (0,-1) {$\mathsf{X+Y}$};
\node[right] (P4) at (1,-1) {$2\mathsf{Y}$};
\node[left]  (P5) at (0,-2) {$\mathsf{Y}$};
\node[right] (P6) at (1,-2) {$\mathsf{0}$};

\draw[->] (P1) to node[above] {\footnotesize $\kappa_1$} (P2);
\draw[->] (P3) to node[above] {\footnotesize $\kappa_2$} (P4);
\draw[->] (P5) to node[above] {\footnotesize $\kappa_3$} (P6);

\node at (8,-1)
{$\begin{aligned}
\dot{x} &= \kappa_1 x^2 - \kappa_2 xy, \\
\dot{y} &= \kappa_2 xy - \kappa_3 y, \\
\end{aligned}$};

\end{tikzpicture}
\end{aligned}
\end{align}
which is called the Lotka--Volterra--Autocatalator (LVA) \cite[Eq.\ (8)]{farkas:noszticzius:1985}, \cite[Eq.\ (1)]{simon:1992}. As shown in \cite{farkas:noszticzius:1985} using a Lyapunov function, the positive equilibrium is globally repelling, and the system can have no periodic orbit. The LVA belongs to a family of networks we refer to as the \lq\lq generalised LVA\rq\rq, namely,
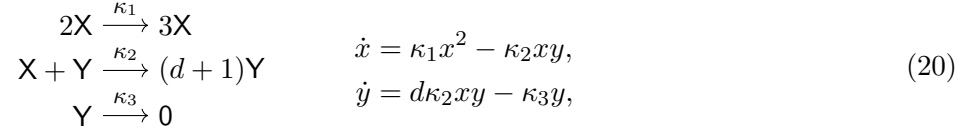
\begin{align}\label{eq:LVA_d}
\begin{aligned}
\begin{tikzpicture}[scale=0.6]

\node[left]  (P1) at (0, 0) {$2\mathsf{X}$};
\node[right] (P2) at (1, 0) {$3\mathsf{X}$};
\node[left]  (P3) at (0,-1) {$\mathsf{X+Y}$};
\node[right] (P4) at (1,-1) {$(d+1)\mathsf{Y}$};
\node[left]  (P5) at (0,-2) {$\mathsf{Y}$};
\node[right] (P6) at (1,-2) {$\mathsf{0}$};

\draw[->] (P1) to node[above] {\footnotesize $\kappa_1$} (P2);
\draw[->] (P3) to node[above] {\footnotesize $\kappa_2$} (P4);
\draw[->] (P5) to node[above] {\footnotesize $\kappa_3$} (P6);

\node at (8,-1)
{$\begin{aligned}
\dot{x} &= \kappa_1 x^2 - \kappa_2 xy, \\
\dot{y} &= d\kappa_2 xy - \kappa_3 y, \\
\end{aligned}$};

\end{tikzpicture}
\end{aligned}
\end{align}
where $d\geq1$. From \eqref{eq:sgn_chain} and \eqref{eq:trJ}, we find that $\det J>0$ and $\tr J>0$ at the unique positive equilibrium of any system in the generalised LVA family, and thus it is a repellor. In fact, by the use of the Dulac function $\frac{1}{xy}$ (as in the proof of \Cref{lem:divergence}), one sees that networks in \eqref{eq:LVA_d} cannot have a periodic orbit.

What makes the generalised LVA unique becomes apparent in \Cref{lem:saddle_or_LVA} below.

\subsection{Proof of \texorpdfstring{\Cref{thm:planar_trimolec}}{}}
\label{subsec:planar_trimolec}

To prove \Cref{thm:planar_trimolec}, consider a two-species, three-reaction, trimolecular, quadratic network. Assume also that the network has rank two, is dynamically nontrivial, and the sources span a triangle (as discussed in \Cref{subsec:planar_setting}, these are all necessary for the associated mass-action differential equation to have a periodic orbit). We now distinguish between two cases. On the one hand, if neither $2\mathsf{X} \longrightarrow 3\mathsf{X}$ nor $2\mathsf{Y} \longrightarrow 3\mathsf{Y}$ is present then, noting that a rank-two network on two species clearly has no trivial species, the system admits no isolated periodic orbit by \Cref{cor:divergence_2X_3X}. On the other hand, if one of these reactions is present then, by \Cref{lem:saddle_or_LVA} below, the system admits no periodic orbit. Hence, the proof of \Cref{thm:planar_trimolec} will be complete after we prove \Cref{lem:saddle_or_LVA}.

\begin{lemma}\label{lem:saddle_or_LVA}
Suppose a $(2,3,2)$ network \eqref{eq:three_reactions_network} includes the reaction $2\mathsf{X} \longrightarrow 3  \mathsf{X}$ and satisfies the conditions above (i.e., it is quadratic, trimolecular, dynamically nontrivial, and has sources spanning a positively oriented triangle). Let $J$ be the Jacobian matrix of the associated mass-action system at the unique positive equilibrium. Then, either
\begin{enumerate}[label = {(\alph*)}]
    \item $\det J<0$ (the equilibrium is a saddle); or
    \item $\det J>0$, and the network belongs to the generalised LVA family \eqref{eq:LVA_d}.
\end{enumerate}
Consequently, the system admits no periodic orbit.
\end{lemma}
\begin{proof} The assumptions of the lemma imply that $c_1=1$ and $d_1=0$ and so the differential equation \eqref{eq:three_reactions_ode} reads 
\begin{align*}
\begin{split}
\dot{x} &= \kappa_1 x^2 + c_2\kappa_2 x^{a_2}y^{b_2} + c_3\kappa_3 x^{a_3}y^{b_3}, \\
\dot{y} &= \phantom{\kappa_1 x^2 +} d_2\kappa_2 x^{a_2}y^{b_2} + d_3\kappa_3 x^{a_3}y^{b_3}.
\end{split}
\end{align*}
If $\det J < 0$, then the system admits no periodic orbit (see \Cref{subsec:saddle_no_periodic}). So now suppose $\det J > 0$. Then $u\in\mathbb{R}^3_+$, see \eqref{eq:sgn_chain}. Since $c_1 = 1$ and $d_1 = 0$, we get, from \eqref{eq:u_formula}, that $u_2 = -d_3 > 0$ and $u_3 =  d_2 > 0$. By \eqref{eq:sources_triangle}, the source $(a_3,b_3)$ is one of $(0,0)$, $(1,0)$ or $(0,1)$, but only in the latter case could $d_3$ be negative. Hence, $(a_3,b_3)=(0,1)$ and $d_3=-1$. Further, $c_3 \geq 0$ also follows. Then $u_1 = c_2d_3 - c_3d_2 > 0$ implies $c_2 < 0$, hence $(a_2, b_2) = (1,1)$ and $c_2 = -1$ (where we once more used \eqref{eq:sources_triangle}). Then $u_1 = 1 - c_3d_2 > 0$ and hence $c_3 = 0$. Consequently, the network belongs to the generalised LVA family \eqref{eq:LVA_d} and, by the discussion in \Cref{subsec:LVA}, the system admits no periodic orbit. 
\end{proof}

\subsection{Discussion on \texorpdfstring{\Cref{thm:planar_trimolec}}{}} \label{subsec:counterex}

By \Cref{thm:ndim_trimolec} below, the conclusion of \Cref{thm:planar_trimolec} holds true for systems with any number of species. However, as we now illustrate by examples, the restrictions on the number of reactions and on the (source and target) molecularities cannot be dropped. Each of the following three planar networks
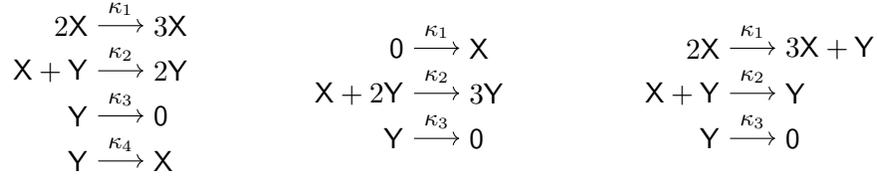
\begin{align*}
    \begin{tikzpicture}[scale=0.6]

\node[left]  (P1) at (0, 0) {$2\mathsf{X}$};
\node[right] (P2) at (1, 0) {$3\mathsf{X}$};
\node[left]  (P3) at (0,-1) {$\mathsf{X+Y}$};
\node[right] (P4) at (1,-1) {$2\mathsf{Y}$};
\node[left]  (P5) at (0,-2) {$\mathsf{Y}$};
\node[right] (P6) at (1,-2) {$\mathsf{0}$};
\node[left]  (P7) at (0,-3) {$\mathsf{Y}$};
\node[right] (P8) at (1,-3) {$\mathsf{X}$};

\draw[->] (P1) to node[above] {\footnotesize $\kappa_1$} (P2);
\draw[->] (P3) to node[above] {\footnotesize $\kappa_2$} (P4);
\draw[->] (P5) to node[above] {\footnotesize $\kappa_3$} (P6);
\draw[->] (P7) to node[above] {\footnotesize $\kappa_4$} (P8);

\begin{scope}[shift={(7,-0.5)}]
\node[left]  (P1) at (0, 0) {$\mathsf{0}$};
\node[right] (P2) at (1, 0) {$\mathsf{X}$};
\node[left]  (P3) at (0,-1) {$\mathsf{X}+2\mathsf{Y}$};
\node[right] (P4) at (1,-1) {$3\mathsf{Y}$};
\node[left]  (P5) at (0,-2) {$\mathsf{Y}$};
\node[right] (P6) at (1,-2) {$\mathsf{0}$};

\draw[->] (P1) to node[above] {\footnotesize $\kappa_1$} (P2);
\draw[->] (P3) to node[above] {\footnotesize $\kappa_2$} (P4);
\draw[->] (P5) to node[above] {\footnotesize $\kappa_3$} (P6);
\end{scope}

\begin{scope}[shift={(14,-0.5)}]
\node[left]  (P1) at (0, 0) {$2\mathsf{X}$};
\node[right] (P2) at (1, 0) {$3\mathsf{X}+\mathsf{Y}$};
\node[left]  (P3) at (0,-1) {$\mathsf{X+Y}$};
\node[right] (P4) at (1,-1) {$\mathsf{Y}$};
\node[left]  (P5) at (0,-2) {$\mathsf{Y}$};
\node[right] (P6) at (1,-2) {$\mathsf{0}$};

\draw[->] (P1) to node[above] {\footnotesize $\kappa_1$} (P2);
\draw[->] (P3) to node[above] {\footnotesize $\kappa_2$} (P4);
\draw[->] (P5) to node[above] {\footnotesize $\kappa_3$} (P6);

\end{scope}

\end{tikzpicture}
\end{align*}
admits a supercritical Andronov--Hopf bifurcation, and thus a stable limit cycle. On the left is a four-reaction, quadratic, trimolecular, mass-action system obtained by adding a reaction to the LVA \eqref{eq:LVA_basic}; in the middle is a three-reaction, cubic, mass-action system known as the Selkov oscillator \cite{selkov:1968}; and on the right is a three-reaction, quadratic, tetramolecular, mass-action system which appeared as \eqref{eq:tetra_simplest}, and is the simplest member of the family \eqref{eq:tetra} below. The bifurcations in these networks occur at $\kappa_3 = \kappa_4$, $\kappa_2=\frac{\kappa_3^3}{\kappa_1^2}$, and $\kappa_1 = \kappa_2$, respectively. The analysis of all three systems is performed in \cite{boros:2022}.




\subsection{Andronov--Hopf bifurcations}\label{subsec:planar_hopf}

This subsection is devoted to finding all three-reaction, planar, quadratic, mass-action systems that admit an \emph{Andronov--Hopf bifurcation}, where a pair of complex conjugate eigenvalues crosses the imaginary axis as a parameter varies, see e.g.\ \cite[Section 7.2]{farkas:1994} or \cite[Section 3.4]{kuznetsov:2004}. In light of \Cref{thm:planar_trimolec}, networks allowing such a bifurcation necessarily have a target complex with molecularity at least four. If the bifurcation is supercritical, a stable limit cycle is born. In degenerate cases (e.g.\ in linear systems \cite[Fig. 3.9]{kuznetsov:2004}), there is a one-parameter family of periodic solutions (a center) at the critical bifurcation parameter, an event we refer to as a \emph{vertical} Andronov--Hopf bifurcation.

We proceed first by enumerating all possible configurations of (at most) bimolecular source complexes on two species. For each configuration we then identify which networks satisfy the necessary conditions for oscillation described above, including that they must be dynamically nontrivial, and conditions on the determinant and trace of the Jacobian matrix at positive equilibria. Where the set of networks with given source complexes and meeting these necessary conditions is nonempty, we identify those which admit an Andronov--Hopf bifurcation, and confirm whether the bifurcation is nondegenerate.

Up to exchange of $\mathsf{X}$ and $\mathsf{Y}$, there are ten ways to choose three bimolecular source complexes that do not lie on a line. These are illustrated in the following diagram:
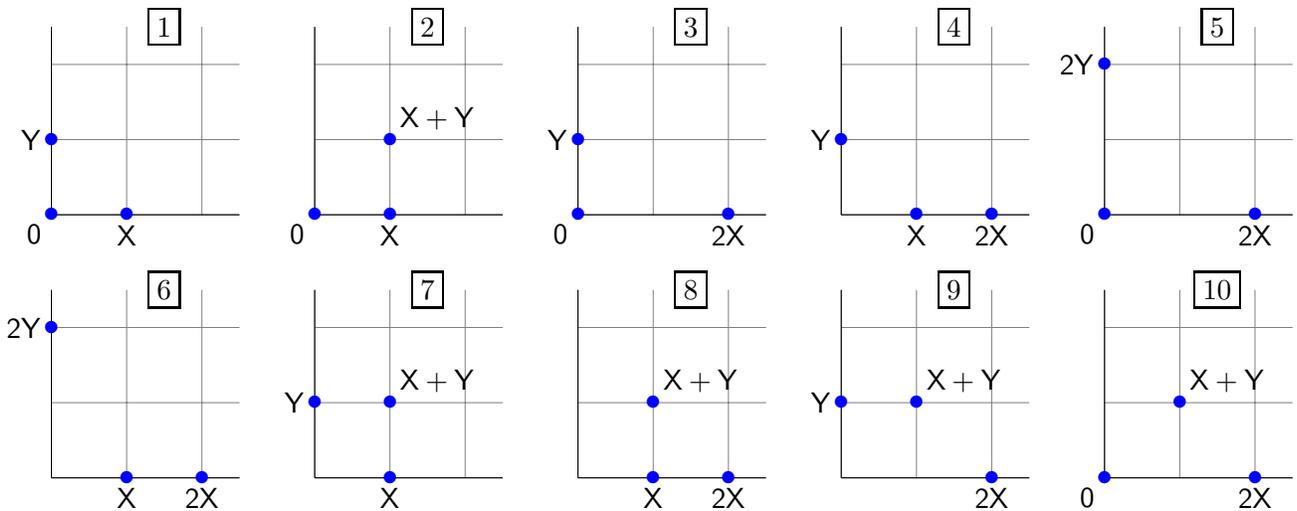
\begin{align*}
\begin{tikzpicture}

\draw [step=1, gray, very thin] (0,0) grid (2.5,2.5);
\draw [ -, black] (0,0)--(2.5,0);
\draw [ -, black] (0,0)--(0,2.5);
\node at (1.5,2.5) {\boxed{1}};
\node (P1) at (1,0) {\large \textcolor{blue}{$\bullet$}};
\node (P2) at (0,1) {\large \textcolor{blue}{$\bullet$}};
\node (P3) at (0,0) {\large \textcolor{blue}{$\bullet$}};
\node [below]      at (P1) {$\mathsf{X}$};
\node [left]       at (P2) {$\mathsf{Y}$};
\node [below left] at (P3) {$\mathsf{0}$};

\begin{scope}[shift={(3.5,0)}]
\draw [step=1, gray, very thin] (0,0) grid (2.5,2.5);
\draw [ -, black] (0,0)--(2.5,0);
\draw [ -, black] (0,0)--(0,2.5);
\node at (1.5,2.5) {\boxed{2}};
\node (P1) at (1,0) {\large \textcolor{blue}{$\bullet$}};
\node (P2) at (1,1) {\large \textcolor{blue}{$\bullet$}};
\node (P3) at (0,0) {\large \textcolor{blue}{$\bullet$}};
\node [below]       at (P1) {$\mathsf{X}$};
\node [above right] at (P2) {$\mathsf{X+Y}$};
\node [below left]  at (P3) {$\mathsf{0}$};
\end{scope}

\begin{scope}[shift={(7,0)}]
\draw [step=1, gray, very thin] (0,0) grid (2.5,2.5);
\draw [ -, black] (0,0)--(2.5,0);
\draw [ -, black] (0,0)--(0,2.5);
\node at (1.5,2.5) {\boxed{3}};
\node (P1) at (2,0) {\large \textcolor{blue}{$\bullet$}};
\node (P2) at (0,1) {\large \textcolor{blue}{$\bullet$}};
\node (P3) at (0,0) {\large \textcolor{blue}{$\bullet$}};
\node [below]      at (P1) {$\mathsf{2X}$};
\node [left]       at (P2) {$\mathsf{Y}$};
\node [below left] at (P3) {$\mathsf{0}$};
\end{scope}

\begin{scope}[shift={(10.5,0)}]
\draw [step=1, gray, very thin] (0,0) grid (2.5,2.5);
\draw [ -, black] (0,0)--(2.5,0);
\draw [ -, black] (0,0)--(0,2.5);
\node at (1.5,2.5) {\boxed{4}};
\node (P1) at (2,0) {\large \textcolor{blue}{$\bullet$}};
\node (P2) at (0,1) {\large \textcolor{blue}{$\bullet$}};
\node (P3) at (1,0) {\large \textcolor{blue}{$\bullet$}};
\node [below] at (P1) {$\mathsf{2X}$};
\node [left]  at (P2) {$\mathsf{Y}$};
\node [below] at (P3) {$\mathsf{X}$};
\end{scope}

\begin{scope}[shift={(14,0)}]
\draw [step=1, gray, very thin] (0,0) grid (2.5,2.5);
\draw [ -, black] (0,0)--(2.5,0);
\draw [ -, black] (0,0)--(0,2.5);
\node at (1.5,2.5) {\boxed{5}};
\node (P1) at (2,0) {\large \textcolor{blue}{$\bullet$}};
\node (P2) at (0,2) {\large \textcolor{blue}{$\bullet$}};
\node (P3) at (0,0) {\large \textcolor{blue}{$\bullet$}};
\node [below]      at (P1) {$\mathsf{2X}$};
\node [left]       at (P2) {$\mathsf{2Y}$};
\node [below left] at (P3) {$\mathsf{0}$};
\end{scope}

\begin{scope}[shift={(0,-3.5)}]
\draw [step=1, gray, very thin] (0,0) grid (2.5,2.5);
\draw [ -, black] (0,0)--(2.5,0);
\draw [ -, black] (0,0)--(0,2.5);
\node at (1.5,2.5) {\boxed{6}};
\node (P1) at (2,0) {\large \textcolor{blue}{$\bullet$}};
\node (P2) at (0,2) {\large \textcolor{blue}{$\bullet$}};
\node (P3) at (1,0) {\large \textcolor{blue}{$\bullet$}};
\node [below] at (P1) {$\mathsf{2X}$};
\node [left]  at (P2) {$\mathsf{2Y}$};
\node [below] at (P3) {$\mathsf{X}$};
\end{scope}

\begin{scope}[shift={(3.5,-3.5)}]
\draw [step=1, gray, very thin] (0,0) grid (2.5,2.5);
\draw [ -, black] (0,0)--(2.5,0);
\draw [ -, black] (0,0)--(0,2.5);
\node at (1.5,2.5) {\boxed{7}};
\node (P1) at (1,0) {\large \textcolor{blue}{$\bullet$}};
\node (P2) at (1,1) {\large \textcolor{blue}{$\bullet$}};
\node (P3) at (0,1) {\large \textcolor{blue}{$\bullet$}};
\node [below]       at (P1) {$\mathsf{X}$};
\node [above right] at (P2) {$\mathsf{X+Y}$};
\node [left]        at (P3) {$\mathsf{Y}$};
\end{scope}

\begin{scope}[shift={(7,-3.5)}]
\draw [step=1, gray, very thin] (0,0) grid (2.5,2.5);
\draw [ -, black] (0,0)--(2.5,0);
\draw [ -, black] (0,0)--(0,2.5);
\node at (1.5,2.5) {\boxed{8}};
\node (P1) at (2,0) {\large \textcolor{blue}{$\bullet$}};
\node (P2) at (1,1) {\large \textcolor{blue}{$\bullet$}};
\node (P3) at (1,0) {\large \textcolor{blue}{$\bullet$}};
\node [below]       at (P1) {$\mathsf{2X}$};
\node [above right] at (P2) {$\mathsf{X+Y}$};
\node [below]       at (P3) {$\mathsf{X}$};
\end{scope}

\begin{scope}[shift={(10.5,-3.5)}]
\draw [step=1, gray, very thin] (0,0) grid (2.5,2.5);
\draw [ -, black] (0,0)--(2.5,0);
\draw [ -, black] (0,0)--(0,2.5);
\node at (1.5,2.5) {\boxed{9}};
\node (P1) at (2,0) {\large \textcolor{blue}{$\bullet$}};
\node (P2) at (1,1) {\large \textcolor{blue}{$\bullet$}};
\node (P3) at (0,1) {\large \textcolor{blue}{$\bullet$}};
\node [below]       at (P1) {$\mathsf{2X}$};
\node [above right] at (P2) {$\mathsf{X+Y}$};
\node [left]        at (P3) {$\mathsf{Y}$};
\end{scope}

\begin{scope}[shift={(14,-3.5)}]
\draw [step=1, gray, very thin] (0,0) grid (2.5,2.5);
\draw [ -, black] (0,0)--(2.5,0);
\draw [ -, black] (0,0)--(0,2.5);
\node at (1.5,2.5) {\boxed{10}};
\node (P1) at (2,0) {\large \textcolor{blue}{$\bullet$}};
\node (P2) at (1,1) {\large \textcolor{blue}{$\bullet$}};
\node (P3) at (0,0) {\large \textcolor{blue}{$\bullet$}};
\node [below]       at (P1) {$\mathsf{2X}$};
\node [above right] at (P2) {$\mathsf{X+Y}$};
\node [below left]  at (P3) {$\mathsf{0}$};
\end{scope}

\end{tikzpicture}
\end{align*}

 \Cref{thm:ten_src_triples} below analyses these ten possibilities, and can be regarded as the second of our main results on three-reaction, planar, quadratic networks. We find that in Cases \boxed{1} to \boxed{6} the mass-action systems admit no periodic orbits. In Cases \boxed{7} to \boxed{10}, we find systems with periodic orbits, but only in Cases \boxed{9} and \boxed{10} are these periodic orbits isolated. Notice that in the statement of the theorem in Cases \boxed{7} to \boxed{10} we list the source complexes such that they span a positively oriented triangle.

\begin{theorem}\label{thm:ten_src_triples}
For a three-reaction, planar, quadratic, mass-action system, the following hold.
\begin{itemize}

\item \textbf{Cases \boxed{1} to \boxed{6}.} The system admits no periodic orbit.

\item \textbf{Case \boxed{7} $\mathsf{X}$, $\mathsf{X+Y}$, $\mathsf{Y}$.} The system has a positive equilibrium which is a center for all $\kappa$ if
\begin{align*}
c_3=0, \ d_1=0, \ \sgn c_1 = -\sgn c_2 = -\sgn d_1 = \sgn d_2 \neq0.
\end{align*}
Otherwise, the system admits no periodic orbit.

\item \textbf{Case \boxed{8} $\mathsf{2X}$, $\mathsf{X+Y}$, $\mathsf{X}$.}
The system has a periodic orbit if and only if
\begin{align*}
c_1>0, \ c_2 = -1, \  c_3>0, \ \ d_1>0, \ d_2=-1, \ d_3\geq0, \ \ \frac{d_3}{c_3} < 1 < \frac{d_1}{c_1},\text{ and }\kappa_1c_1+\kappa_2d_2 = 0.
\end{align*}
In this case, the system has a positive equilibrium which is a center, and by varying the ratio $\frac{\kappa_1}{\kappa_2}$, one obtains a vertical Andronov--Hopf bifurcation.

\item \textbf{Case \boxed{9} $\mathsf{2X}$, $\mathsf{X+Y}$, $\mathsf{Y}$.} The system admits an Andronov--Hopf bifurcation if and only if 
\begin{align*}
&c_1>0, \ c_2 = -1,  \ c_3>  0, \ \ d_1>0, \ d_2\geq -1, \ d_3\geq-1, \ \ \frac{1}{2} \left(\frac{d_3}{c_3}+\frac{d_1}{c_1}\right) < \frac{d_2}{c_2} < \frac{d_1}{c_1} \text{ or}\\
&c_1>0, \ c_2 = -1,  \ c_3=0, \ \ d_1>0, \ d_2\geq -1, \ d_3=-1, \ \  \frac{d_2}{c_2} < \frac{d_1}{c_1}.
\end{align*}
Furthermore, the Andronov--Hopf bifurcation is supercritical.

\item \textbf{Case \boxed{10} $\mathsf{2X}$, $\mathsf{X+Y}$, $\mathsf{0}$.} The system admits an Andronov--Hopf bifurcation if and only if 
\begin{align*}
c_1>0, \ c_2 = -1, \ c_3> 0, \ \ d_1>0, \ d_2= -1, \ d_3\geq0, \ \ \frac{1}{2} \left(\frac{d_3}{c_3}+\frac{d_1}{c_1}\right) < \frac{d_2}{c_2} < \frac{d_1}{c_1}.
\end{align*}
Furthermore, the Andronov--Hopf bifurcation is supercritical.
\end{itemize}
\end{theorem}
\begin{proof}
We note first that whenever the networks in cases \boxed{1} to \boxed{10} are dynamically nontrivial, they are nondegenerate and, by \Cref{lem:n_n+1_n_equil}, each corresponding mass-action system has a unique positive equilibrium for any choice of rate constants. On the other hand, if a network is dynamically trivial, then the corresponding mass-action system can have no periodic orbits.
\begin{itemize}
\item \textbf{Cases \boxed{1} to \boxed{6}.}
Assume, by way of contradiction, that the system admits a periodic orbit. First, we show that no reaction of the form \eqref{eq:rxn_pos_div} is present. Then we argue that the differential equation is not of the form \eqref{eq:divergence_odes}, and thus arrive at a contradiction.

In cases \boxed{1} and \boxed{2}, the network obviously cannot have a reaction of the form \eqref{eq:rxn_pos_div}.

In cases \boxed{3} to \boxed{6}, let the first source be $2\mathsf{X}$ and the second source be $\mathsf{Y}$ (cases \boxed{3} and \boxed{4}) or $2\mathsf{Y}$ (cases \boxed{5} and \boxed{6}). In all cases, $a_2=b_1=b_3=0$, and hence, $c_2, d_1, d_3 \geq 0$. Since there is a periodic orbit by assumption, the network must be dynamically nontrivial, and $d_2<0$ follows. Additionally, the Jacobian determinant evaluated at the unique positive equilibrium must be positive, and therefore $u\in\mathbb{R}^3_+$. Hence, $0<u_3=c_1 d_2 - c_2 d_1$ (see \eqref{eq:u_formula}), which implies $c_1<0$. Since $c_1$ and $d_2$ are both negative, indeed there is no reaction of the form \eqref{eq:rxn_pos_div}.

Then, by \Cref{lem:divergence}, the associated mass-action system is of the Lotka--Volterra type with no diagonal term. In particular, only reactions with a source $\mathsf{Y}$ or $\mathsf{X+Y}$ can contribute to $\dot{y}$. However, since in all six cases there is at most one such source, $\dot{y} \geq 0$ or $\dot{y} \leq 0$. This contradicts the existence of a periodic orbit.

\item \textbf{Case \boxed{7} $\mathsf{X}$, $\mathsf{X+Y}$, $\mathsf{Y}$.}
The differential equation reads
\begin{align*}
\begin{split}
    \dot{x} &= c_1\kappa_1 x + c_2\kappa_2 xy + c_3\kappa_3 y, \\
    \dot{y} &= d_1\kappa_1 x + d_2\kappa_2 xy + d_3\kappa_3 y
\end{split}
\end{align*}
with $c_3,d_1\geq0$. After multiplying by the Dulac function $(xy)^{-1}$, the divergence is $-c_3\kappa_3 x^{-2} - d_1\kappa_1 y^{-2}$, which is nonpositive, and zero only if $c_3 = d_1 = 0$. Hence, if $c_3 = d_1 = 0$ is violated then, by the Bendixson--Dulac Test, the system admits no periodic orbit. The analysis of the case $c_3 = d_1 = 0$ is standard: see the proofs of \Cref{lem:divergence} and \Cref{lem:two_param_families}.

\item \textbf{Case \boxed{8} $\mathsf{2X}$, $\mathsf{X+Y}$, $\mathsf{X}$.}
After division by $x$, the differential equation reads    
\begin{align*}
    \dot{x} &= c_1\kappa_1 x + c_2\kappa_2 y + c_3\kappa_3,  \\
    \dot{y} &= d_1\kappa_1 x + d_2\kappa_2 y + d_3\kappa_3 
\end{align*}
with $d_1 \geq 0$ and $d_3 \geq 0$ (since $b_1=0$ and $b_3=0$).
This linear system has a periodic orbit (in fact, a center) if and only if there is a unique positive equilibrium at which $\det J>0$, and $\tr J=0$.

Assume the system admits a periodic orbit. 
Then $d_2<0$ (otherwise $\dot{y}\geq0$ in $\mathbb{R}^2_+$). Since $\tr J = c_1\kappa_1+d_2\kappa_2$, we find that $c_1>0$. Hence, from $u_3>0$ (see \eqref{eq:sgn_chain}) we obtain $c_2d_1<0$, which implies $d_1>0$ and $c_2<0$. From $u_2>0$ it follows that $c_3>0$. In fact, since additionally $c_2, d_2\geq-1$ (because $a_2=b_2=1$), both $c_2$ and $d_2$ equal to $-1$. This gives all the sign conditions on $c_i$, $d_i$ in the statement of the theorem. The inequalities $\frac{d_3}{c_3} < 1 < \frac{d_1}{c_1}$ are equivalent to $u_1$, $u_2$, $u_3>0$. Therefore the conditions in the statement are necessary and sufficient for the existence of a periodic solution.

Clearly, by varying the ratio $\frac{\kappa_1}{\kappa_2}$, one obtains a vertical Andronov--Hopf bifurcation at $\frac{\kappa_1}{\kappa_2}=-\frac{d_2}{c_1}$.

\item \textbf{Case \boxed{9} $\mathsf{2X}$, $\mathsf{X+Y}$, $\mathsf{Y}$.}
For an Andronov--Hopf bifurcation in \eqref{eq:three_reactions_ode}, there must exist a unique positive equilibrium $(\bar{x},\bar{y})$ at which $\det J>0$ and $\tr J=0$ hold.
By \eqref{eq:trJ}, $\tr J = \mu\left(\frac{1}{\bar{x}}(c_1u_1- c_3u_3)-\frac{1}{\bar{y}}d_1u_1\right)$, where we used the general observations that $\sum_{i=1}^3c_iu_i = 0$ and $\sum_{i=1}^3d_iu_i = 0$ (recall that $u=c\times d$).

Case $d_1 = 0$. In this case, $\tr J = 0$ if and only if $c_1u_1 = c_3u_3$. Thus, $c_1>0$ and $c_3>0$ ($c_3\geq0$ is given, since $a_3=0$). Thus, $c_2=-1$ (because $c_2\geq-1$ and $c_2<0$). Also, $d_3 = -1$, because $u_2>0$ implies $d_3<0$. Also, $d_2>0$, because $u_3>0$. Hence, $0 < c_3d_2 < 1$, because $u_1>0$. However, this has no integer solution.

Case $d_1>0$. Then $c_1u_1 > c_3u_3$, and hence, $c_1>0$ and $c_2<0$ (in fact, $c_2=-1$). If $c_3=0$, we find that $u_i>0$ is equivalent to $d_1>0$, $d_3<0$, and $\frac{d_2}{c_2} < \frac{d_1}{c_1}$. If $c_3>0$, we find that $u_i>0$ and $c_1u_1 > c_3u_3$ are equivalent to $\frac{1}{2} \left(\frac{d_3}{c_3}+\frac{d_1}{c_1}\right) < \frac{d_2}{c_2} < \frac{d_1}{c_1}$.

To see that the Andronov--Hopf bifurcation is nondegenerate and, in fact, supercritical, one computes the first focal value and finds it is negative. See the Mathematica Notebook \cite{boros:2022} for the calculations.

\item \textbf{Case \boxed{10} $\mathsf{2X}$, $\mathsf{X+Y}$, $\mathsf{0}$.} As in the proof of case \boxed{9}, for an Andronov--Hopf bifurcation to occur there must exist a unique positive equilibrium $(\bar{x},\bar{y})$ at which $\det J>0$ and $\tr J=0$ hold.

By \eqref{eq:trJ}, $\tr J = \mu\left(\frac{1}{\bar{x}}(c_1u_1- c_3u_3)+\frac{1}{\bar{y}}d_2u_2\right)$, where we used $\sum_{i=1}^3c_iu_i = 0$ (recall that $u=c\times d$).

Since at least one of $d_1$, $d_2$ or $d_3$ must be negative (as the network must be dynamically nontrivial), we find $d_2 = -1$. Then $c_1u_1 > c_3u_3$ implies $c_1>0$. Hence, $c_2 = -1$. Since $u_1>0$, we also find $c_3>0$. Then, as in the proof of case \boxed{9}, we find that the system admits an Andronov--Hopf bifurcation if and only if $\frac{1}{2} \left(\frac{d_3}{c_3}+\frac{d_1}{c_1}\right) < \frac{d_2}{c_2} < \frac{d_1}{c_1}$ holds.

That the first focal value is negative is shown again in  \cite{boros:2022}. Hence, the Andronov--Hopf bifurcation is  supercritical.
\end{itemize}
\end{proof}

\subsection{Realisations of the oscillating systems in \texorpdfstring{\Cref{thm:ten_src_triples}}{}}

In case \boxed{7}, up to exchange of $\mathsf{X}$ and $\mathsf{Y}$, the mass-action systems with a center for all $\kappa$ are given by
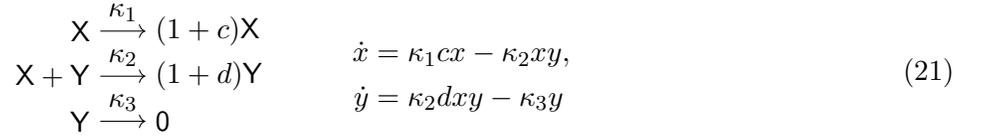
\begin{align}\label{eq:case_7}
\begin{aligned}
\begin{tikzpicture}[scale=0.6]

\node[left]  (P1) at (0, 0) {$\mathsf{X}$};
\node[right] (P2) at (1, 0) {$(1+c)\mathsf{X}$};
\node[left]  (P3) at (0,-1) {$\mathsf{X+Y}$};
\node[right] (P4) at (1,-1) {$(1+d)\mathsf{Y}$};
\node[left]  (P5) at (0,-2) {$\mathsf{Y}$};
\node[right] (P6) at (1,-2) {$\mathsf{0}$};

\draw[->] (P1) to node[above] {$\kappa_1$} (P2);
\draw[->] (P3) to node[above] {$\kappa_2$} (P4);
\draw[->] (P5) to node[above] {$\kappa_3$} (P6);

\node at (8,-1)
{$\begin{aligned}
\dot{x} &= \kappa_1 c x - \kappa_2 xy, \\
\dot{y} &= \kappa_2 d xy - \kappa_3 y \\
\end{aligned}$};

\end{tikzpicture}
\end{aligned}
\end{align}
with $c,d\geq1$. Note that this is one of the families of networks, namely \eqref{eq:two_param_families_left}, obtained in \Cref{lem:two_param_families}, and this family also appears in \cite[Eq. (5)]{farkas:noszticzius:1985}. For $c=d=1$ the network is bimolecular and is, in fact, the Lotka reactions \eqref{eq:lotka_only}. Since the qualitative picture is the same for all $\kappa$, these systems admit no bifurcation as $\kappa$ is varied.

In case \boxed{8}, no tetramolecular network admits a vertical Andronov--Hopf bifurcation. The pentamolecular examples are
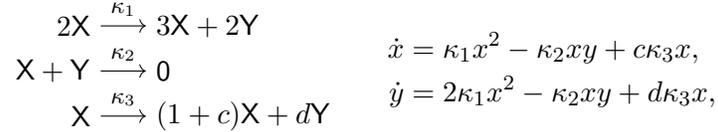
\begin{align*}
\begin{tikzpicture}[scale=0.6]

\node[left]  (P1) at (0, 0) {$2\mathsf{X}$};
\node[right] (P2) at (1, 0) {$3\mathsf{X}+2\mathsf{Y}$};
\node[left]  (P3) at (0,-1) {$\mathsf{X+Y}$};
\node[right] (P4) at (1,-1) {$\mathsf{0}$};
\node[left]  (P5) at (0,-2) {$\mathsf{X}$};
\node[right] (P6) at (1,-2) {$(1+c)\mathsf{X}+d\mathsf{Y}$};

\draw[->] (P1) to node[above] {\footnotesize $\kappa_1$} (P2);
\draw[->] (P3) to node[above] {\footnotesize $\kappa_2$} (P4);
\draw[->] (P5) to node[above] {\footnotesize $\kappa_3$} (P6);

\node at (10,-1)
{$\begin{aligned}
\dot{x} &= \kappa_1 x^2  - \kappa_2 xy + c\kappa_3 x, \\
\dot{y} &= 2\kappa_1 x^2 - \kappa_2 xy + d\kappa_3 x, \\
\end{aligned}$};

\end{tikzpicture}
\end{align*}
where $0\leq d<c$ and $c+d\leq 4$.

The tetramolecular networks in case \boxed{9} that admit an Andronov--Hopf bifurcation are
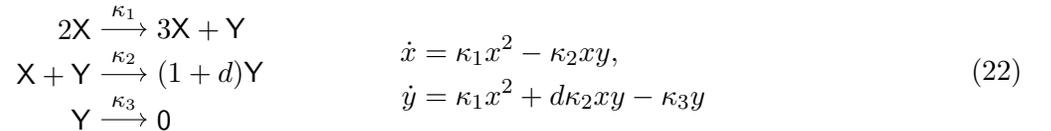
\begin{align}\label{eq:tetra}
\begin{aligned}
\begin{tikzpicture}[scale=0.6]

\node[left]  (P1) at (0, 0) {$2\mathsf{X}$};
\node[right] (P2) at (1, 0) {$3\mathsf{X}+\mathsf{Y}$};
\node[left]  (P3) at (0,-1) {$\mathsf{X+Y}$};
\node[right] (P4) at (1,-1) {$(1+d)\mathsf{Y}$};
\node[left]  (P5) at (0,-2) {$\mathsf{Y}$};
\node[right] (P6) at (1,-2) {$\mathsf{0}$};

\draw[->] (P1) to node[above] {\footnotesize $\kappa_1$} (P2);
\draw[->] (P3) to node[above] {\footnotesize $\kappa_2$} (P4);
\draw[->] (P5) to node[above] {\footnotesize $\kappa_3$} (P6);

\node at (10,-1)
{$\begin{aligned}
\dot{x} &= \kappa_1 x^2 - \kappa_2 xy, \\
\dot{y} &= \kappa_1 x^2 + d\kappa_2 xy - \kappa_3 y \\
\end{aligned}$};

\end{tikzpicture}
\end{aligned}
\end{align}
with $d=0,1,2,3$.

In case \boxed{10}, no hexamolecular network admits an Andronov--Hopf bifurcation. The heptamolecular examples are
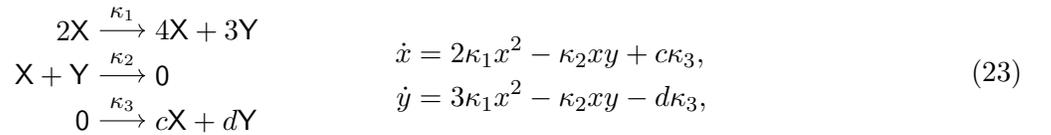
\begin{align}\label{eq:hepta}
\begin{aligned}
\begin{tikzpicture}[scale=0.6]

\node[left]  (P1) at (0, 0) {$2\mathsf{X}$};
\node[right] (P2) at (1, 0) {$4\mathsf{X}+3\mathsf{Y}$};
\node[left]  (P3) at (0,-1) {$\mathsf{X+Y}$};
\node[right] (P4) at (1,-1) {$\mathsf{0}$};
\node[left]  (P5) at (0,-2) {$\mathsf{0}$};
\node[right] (P6) at (1,-2) {$c\mathsf{X}+d\mathsf{Y}$};

\draw[->] (P1) to node[above] {\footnotesize $\kappa_1$} (P2);
\draw[->] (P3) to node[above] {\footnotesize $\kappa_2$} (P4);
\draw[->] (P5) to node[above] {\footnotesize $\kappa_3$} (P6);

\node at (10,-1)
{$\begin{aligned}
\dot{x} &= 2\kappa_1 x^2 - \kappa_2 xy + c \kappa_3, \\
\dot{y} &= 3\kappa_1 x^2 - \kappa_2 xy - d \kappa_3, \\
\end{aligned}$};

\end{tikzpicture}
\end{aligned}
\end{align}
where $c>0$, $d\geq0$, $c+d \leq 7$ and $\frac{d}{c}<\frac12$.
\section{The analysis of three-reaction, quadratic, trimolecular systems}\label{sec:n_3_2}

Our main goal in this section is to extend \Cref{thm:planar_trimolec} to an arbitrary number of species, thereby obtaining \Cref{thm:intro_trimolec}, namely, no three-reaction, quadratic, trimolecular, mass-action system has an isolated periodic orbit. In fact, we prove more in \Cref{thm:ndim_trimolec}, where mass-action systems in this class which admit a periodic orbit are fully characterized.

Recall from \Cref{subsec:MA} that a three-reaction mass-action system which admits a periodic orbit must have rank two. We thus focus on $(n,3,2)$ systems. In \Cref{subsec:fold}, we demonstrate by an example that $(n,3,2)$ mass-action systems with $n\geq3$ can have multiple isolated positive equilibria in a stoichiometric class, a phenomenon that does not occur when $n=2$, and makes the general case slightly more complicated. In \Cref{subsec:LVA_lifted}, we study a family of networks that is related to the generalised LVA \eqref{eq:LVA_d}, and which plays an important role in \Cref{thm:ndim_trimolec}. In \Cref{subsec:main_result}, we state and prove our main result, \Cref{thm:ndim_trimolec}. The proof uses three lemmas that are stated and proved in \Cref{subsec:2X_3X}.

\subsection{Number of equilibria} \label{subsec:fold}

For $(2,3,2)$ networks, by \Cref{lem:n_n+1_n_equil}, whenever there is an isolated positive equilibrium, there is exactly one positive equilibrium for each choice of $\kappa$. For $(n,3,2)$ networks with $n\geq3$, the number of isolated positive equilibria may depend on the positive stoichiometric class in question. In fact, even if an equilibrium is unique on its stoichiometric class, it can be degenerate. To illustrate these statements, consider the bimolecular $(3,3,2)$ mass-action system
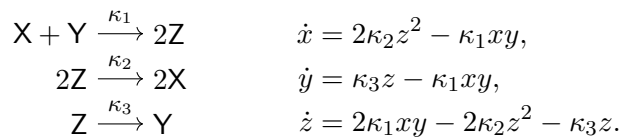
\begin{align*}
\begin{tikzpicture}[scale=0.6]

\node[left]  (P1) at (0, 0) {$\mathsf{X}+\mathsf{Y}$};
\node[right] (P2) at (1, 0) {$2\mathsf{Z}$};
\node[left]  (P3) at (0,-1) {$2\mathsf{Z}$};
\node[right] (P4) at (1,-1) {$2\mathsf{X}$};
\node[left]  (P5) at (0,-2) {$\mathsf{Z}$};
\node[right] (P6) at (1,-2) {$\mathsf{Y}$};

\draw[->] (P1) to node[above] {\footnotesize $\kappa_1$} (P2);
\draw[->] (P3) to node[above] {\footnotesize $\kappa_2$} (P4);
\draw[->] (P5) to node[above] {\footnotesize $\kappa_3$} (P6);

\node at (8,-1)
{$\begin{aligned}
\dot{x} &= 2 \kappa_2 z^2 - \kappa_1 x y, \\
\dot{y} &= \kappa_3 z     - \kappa_1 x y, \\
\dot{z} &= 2 \kappa_1 x y - 2 \kappa_2 z^2 - \kappa_3 z.
\end{aligned}$};

\end{tikzpicture}
\end{align*}
The set of positive equilibria is the hyperbola $\{(x,y,z) \in \mathbb{R}^3_+ \colon xy=\frac{\kappa_3^2}{2\kappa_1\kappa_2}, z = \frac{\kappa_3}{2\kappa_2}\}$, and the positive stoichiometric classes are $\mathcal{P}_C=\{(x,y,z)\in\mathbb{R}^3_+\colon x+y+z=C\}$ for $C>0$. For any fixed rate constants, the number of positive equilibria in $\mathcal{P}_C$ is $0$, $1$, or $2$, depending on $C$. In fact, the system admits fold bifurcations of equilibria, a phenomenon that is ruled out for $(2,3,2)$ mass-action systems.

\subsection{The Lifted LVA} \label{subsec:LVA_lifted}

The networks we study in this subsection play a special role in the main result, \Cref{thm:ndim_trimolec}.
We consider the family of mass-action systems
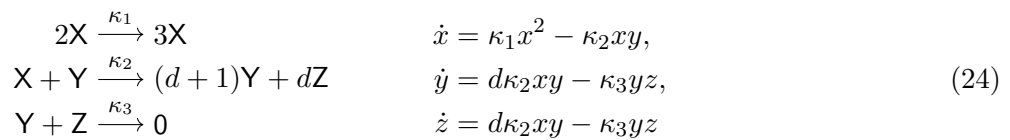
\begin{align}\label{eq:LVA_d_lifted}
\begin{aligned}
\begin{tikzpicture}[scale=0.6]

\node[left]  (P1) at (0, 0) {$2\mathsf{X}$};
\node[right] (P2) at (1, 0) {$3\mathsf{X}$};
\node[left]  (P3) at (0,-1) {$\mathsf{X+Y}$};
\node[right] (P4) at (1,-1) {$(d+1)\mathsf{Y}+d\mathsf{Z}$};
\node[left]  (P5) at (0,-2) {$\mathsf{Y+Z}$};
\node[right] (P6) at (1,-2) {$\mathsf{0}$};

\draw[->] (P1) to node[above] {\footnotesize $\kappa_1$} (P2);
\draw[->] (P3) to node[above] {\footnotesize $\kappa_2$} (P4);
\draw[->] (P5) to node[above] {\footnotesize $\kappa_3$} (P6);

\node at (10,-1)
{$\begin{aligned}
\dot{x} &= \kappa_1 x^2 - \kappa_2 xy, \\
\dot{y} &= d\kappa_2 xy - \kappa_3 yz, \\
\dot{z} &= d\kappa_2 xy - \kappa_3 yz
\end{aligned}$};

\end{tikzpicture}
\end{aligned}
\end{align}
with $d\geq1$. Each member of this family is obtained by adding a new species, $\mathsf{Z}$, to networks in the generalised LVA family \eqref{eq:LVA_d} while preserving the rank of the network. The case $d=1$ corresponds to the Lifted LVA \eqref{eq:LVA_lifted}, and consequently we refer to the networks in \eqref{eq:LVA_d_lifted} as the \lq\lq Lifted LVA family\rq\rq. Note that the Lifted LVA is the only trimolecular network in this family. 

Observe that as the networks in the Lifted LVA family clearly have no trivial species, any periodic orbits of the corresponding mass-action systems must be positive (see \Cref{subsec:saddle_no_periodic}), and we can restrict attention to the positive orthant. The set of positive equilibria of \eqref{eq:LVA_d_lifted} is the ray $\{(t, \frac{\kappa_1}{\kappa_2}t, \frac{\kappa_2d}{\kappa_3}t)\colon t>0\}$. Since $\dot y = \dot z$, the stoichiometric classes are given by the planes $z = y + C$. Hence, in local coordinates on a positive stoichiometric class, the differential equation \eqref{eq:LVA_d_lifted} reduces to the 2d Lotka--Volterra system
\begin{align}
\label{eq:lifted_LVA_2d_LV}
\begin{split}
\dot{x} &= x(\kappa_1 x -\kappa_2 y), \\
\dot{y} &= y(d\kappa_2 x - \kappa_3 (y+C)).
\end{split}
\end{align}
To analyse \eqref{eq:lifted_LVA_2d_LV}, we apply the coordinate transformation
$v= \frac{y}{x}$, $w=\frac{1}{x}$, which takes $\mathbb{R}^2_+$ to $\mathbb{R}^2_+$ and,
after a multiplication by $w$, transforms \eqref{eq:lifted_LVA_2d_LV} into the Lotka--Volterra system 
\begin{align}
\label{eq:lifted_LVA_pred_prey}
\begin{split}
\dot{v} &= v((d\kappa_2 - \kappa_1) + (\kappa_2 -\kappa_3) v - \kappa_3 C w), \\
\dot{w} &= w(-\kappa_1 + \kappa_2 v) .
\end{split}
\end{align}
\Cref{eq:lifted_LVA_pred_prey} has a positive equilibrium if and only if $C > 0$ and $d\kappa_2^2 >  \kappa_1 \kappa_3$, and in this case the positive equilibrium, $(\bar v, \bar w)$, is unique, and the Jacobian determinant of the system evaluated at this equilibrium is positive. For $C>0$, the predator--prey system \eqref{eq:lifted_LVA_pred_prey} has a  Lyapunov function
$V(v,w) = \kappa_2 (v-\bar v \log v) + \kappa_3 C (w- \bar w \log w)$, which is convex, attains its unique minimum at $(\bar v, \bar w)$, and satisfies $\dot V = \kappa_2 (\kappa_2 -\kappa_3)(v - \bar v)^2$, see e.g.\ \cite[Section 2.7]{hofbauer:sigmund:1998}. 
Therefore, for $\kappa_2 < \kappa_3$, the positive equilibrium is  globally asymptotically stable (on $\mathbb{R}^2_+$), for $\kappa_2 = \kappa_3$ it is a global center, and for $\kappa_2 > \kappa_3$, it is a global repellor. Therefore systems in the Lifted LVA family \eqref{eq:LVA_d_lifted} undergo a vertical Andronov--Hopf bifurcation, as 
$\frac{\kappa_2}{\kappa_3}$ increases through the value 1, provided $\kappa_1 < d\kappa_2$. This happens simultaneously in all stoichiometric classes with $C > 0$. In particular, isolated periodic orbits cannot occur in this family. 

\subsection{Main result} \label{subsec:main_result}

We are now in the position to state our main result. Apart from the results already proved, the proof also relies on three lemmas which are stated and proved in \Cref{subsec:2X_3X}.

\begin{theorem}\label{thm:ndim_trimolec}
Assume that a three-reaction, quadratic, trimolecular, mass-action system with no trivial species has a periodic orbit. Then, up to a permutation of the species, the differential equation is one of the following.
\begin{align*}
\begin{tikzpicture}

\node at (0,1/2) {(I)};
\node[below] at (0,0)
{$\begin{aligned}
\dot{x} &= x(\kappa_1 c - \kappa_2 y) \\
\dot{y} &= y(\kappa_2 d x-\kappa_3)
\end{aligned}$};
\node at (0,-2.25) {$\begin{array}{c} \kappa_1, \kappa_2, \kappa_3>0 \\ c,d \in \{1,2\}\end{array}$};

\node at (5,1/2) {(II)};
\node[below] at (5,0)
{$\begin{aligned}
    \dot{x} &= x(\kappa_1 z-\kappa_2 y) \\
    \dot{y} &= y(\kappa_2 x-\kappa_3 z) \\
    \dot{z} &= z(\kappa_3 y-\kappa_1 x)
\end{aligned}$};
\node at (5,-2.25) {$\kappa_1, \kappa_2, \kappa_3>0$};

\node at (10,1/2) {(III)};
\node[below] at (10,0)
{$\begin{aligned}
    \dot{x} &= x(\kappa_1 x - \kappa_2 y) \\
    \dot{y} &= y(\kappa_2 x - \kappa_2 z) \\
    \dot{z} &= y(\kappa_2 x - \kappa_2 z)
\end{aligned}$};
\node at (10,-2.25) {$0<\kappa_1<\kappa_2$};

\end{tikzpicture}
\end{align*}
Note that
\begin{itemize}
    \item equation (I) is the generalised Lotka ODE \eqref{eq:case_7}, and the unique positive equilibrium is a global center,
    \item equation (II) is the Ivanova ODE \eqref{eq:ivanova_only}, and in each positive stoichiometric class $x+y+z=C>0$ the unique positive equilibrium is a global center,
    \item equation (III) is the Lifted LVA ODE \eqref{eq:LVA_lifted} with $\kappa_2=\kappa_3>\kappa_1$, and the unique positive equilibrium in each positive stoichiometric class $z-y=C>0$ is a global center, while there is no positive equilibrium in the positive stoichiometric classes $z-y=C\leq0$.
\end{itemize}
In particular, mass-action systems satisfying the conditions of the theorem admit no isolated periodic orbits.
\end{theorem}
\begin{proof}
Recall from \Cref{subsec:MA} that a three-reaction network, whose mass-action system admits a periodic orbit, has rank two. Hence, by remarks in \Cref{subsec:saddle_no_periodic}, the network must be dynamically nontrivial, and any periodic orbit of a system satisfying the hypotheses of the theorem must be positive. We distinguish between two cases.
\begin{itemize}
    \item In case no reaction in the network is of the form $\mathsf{2X}_i \to \mathsf{3X}_i$, by \Cref{lem:two_param_families} we find that the network is either the generalised Lotka reactions \eqref{eq:two_param_families_left} with $c,d \in \{1,2\}$, leading to (I), or the Ivanova reactions, leading to (II). Note that none of the networks in \eqref{eq:two_param_families_right} is trimolecular.
    \item Assume now that the reaction $\mathsf{2X} \to \mathsf{3X}$ is present. In case there are three species, the differential equation is (III) by \Cref{lem:saddle_or_lifted_LVA} below and the discussion in \Cref{subsec:LVA_lifted}. In case there are two, four, or at least five species then the system admits no periodic orbit by \Cref{lem:saddle_or_LVA,lem:4species_saddle,lem:5species}, respectively.
\end{itemize}
\end{proof}

Notice that \Cref{thm:intro_trimolec} is a corollary of \Cref{thm:ndim_trimolec}. Furthermore, \Cref{thm:ndim_trimolec} is an extension of \Cref{thm:planar_trimolec} to an arbitrary number of species. In fact, even though all networks in \Cref{thm:ndim_trimolec} are assumed to have no trivial species, the conclusion that three-reaction, quadratic, trimolecular, mass-action systems admit no isolated periodic orbits clearly holds true without this assumption. Based on \Cref{thm:ndim_trimolec} we can list all three-reaction, quadratic, trimolecular $(n,3,2)$ networks (including those with some trivial species), whose mass-action system admits a periodic orbit: there are $16$ such networks, see \Cref{fig:16_networks}.

\begin{figure}[h!t!]
\begin{center}
\begin{tikzpicture}[scale=0.6]

\node[left]  (P1) at (0, 0)    {$\mathsf{X}$};
\node[right] (P2) at (0.75, 0) {$2\mathsf{X}$};
\node[left]  (P3) at (0,-1)    {$\mathsf{X+Y}$};
\node[right] (P4) at (0.75,-1) {$2\mathsf{Y}$};
\node[left]  (P5) at (0,-2)    {$\mathsf{Y}$};
\node[right] (P6) at (0.75,-2) {$\mathsf{0}$};
\draw[->] (P1) to node[above] {} (P2);
\draw[->] (P3) to node[above] {} (P4);
\draw[->] (P5) to node[above] {} (P6);

\begin{scope}[shift={(0,-4)}]
\node[left]  (P1) at (0, 0)    {$\mathsf{X}$};
\node[right] (P2) at (0.75, 0) {$2\mathsf{X}$};
\node[left]  (P3) at (0,-1)    {$\mathsf{X+Y}$};
\node[right] (P4) at (0.75,-1) {$3\mathsf{Y}$};
\node[left]  (P5) at (0,-2)    {$\mathsf{Y}$};
\node[right] (P6) at (0.75,-2) {$\mathsf{0}$};
\draw[->] (P1) to node[above] {} (P2);
\draw[->] (P3) to node[above] {} (P4);
\draw[->] (P5) to node[above] {} (P6);
\end{scope}

\begin{scope}[shift={(6,0)}]
\node[left]  (P1) at (0, 0) {$\mathsf{X}$};
\node[right] (P2) at (0.75, 0) {$2\mathsf{X}$};
\node[left]  (P3) at (0,-1) {$\mathsf{X+Y}$};
\node[right] (P4) at (0.75,-1) {$2\mathsf{Y}$};
\node[left]  (P5) at (0,-2) {$\mathsf{Y+Z}$};
\node[right] (P6) at (0.75,-2) {$\mathsf{Z}$};
\draw[->] (P1) to node[above] {} (P2);
\draw[->] (P3) to node[above] {} (P4);
\draw[->] (P5) to node[above] {} (P6);
\end{scope}

\begin{scope}[shift={(6,-4)}]
\node[left]  (P1) at (0, 0) {$\mathsf{X}$};
\node[right] (P2) at (0.75, 0) {$2\mathsf{X}$};
\node[left]  (P3) at (0,-1) {$\mathsf{X+Y}$};
\node[right] (P4) at (0.75,-1) {$3\mathsf{Y}$};
\node[left]  (P5) at (0,-2) {$\mathsf{Y+Z}$};
\node[right] (P6) at (0.75,-2) {$\mathsf{Z}$};
\draw[->] (P1) to node[above] {} (P2);
\draw[->] (P3) to node[above] {} (P4);
\draw[->] (P5) to node[above] {} (P6);
\end{scope}

\begin{scope}[shift={(12,0)}]
\node[left]  (P1) at (0, 0) {$\mathsf{X}$};
\node[right] (P2) at (0.75, 0) {$3\mathsf{X}$};
\node[left]  (P3) at (0,-1) {$\mathsf{X+Y}$};
\node[right] (P4) at (0.75,-1) {$2\mathsf{Y}$};
\node[left]  (P5) at (0,-2) {$\mathsf{Y}$};
\node[right] (P6) at (0.75,-2) {$\mathsf{0}$};
\draw[->] (P1) to node[above] {} (P2);
\draw[->] (P3) to node[above] {} (P4);
\draw[->] (P5) to node[above] {} (P6);
\end{scope}

\begin{scope}[shift={(12,-4)}]
\node[left]  (P1) at (0, 0) {$\mathsf{X}$};
\node[right] (P2) at (0.75, 0) {$3\mathsf{X}$};
\node[left]  (P3) at (0,-1) {$\mathsf{X+Y}$};
\node[right] (P4) at (0.75,-1) {$3\mathsf{Y}$};
\node[left]  (P5) at (0,-2) {$\mathsf{Y}$};
\node[right] (P6) at (0.75,-2) {$\mathsf{0}$};
\draw[->] (P1) to node[above] {} (P2);
\draw[->] (P3) to node[above] {} (P4);
\draw[->] (P5) to node[above] {} (P6);
\end{scope}

\begin{scope}[shift={(18,0)}]
\node[left]  (P1) at (0, 0) {$\mathsf{X}$};
\node[right] (P2) at (0.75, 0) {$3\mathsf{X}$};
\node[left]  (P3) at (0,-1) {$\mathsf{X+Y}$};
\node[right] (P4) at (0.75,-1) {$2\mathsf{Y}$};
\node[left]  (P5) at (0,-2) {$\mathsf{Y+Z}$};
\node[right] (P6) at (0.75,-2) {$\mathsf{Z}$};
\draw[->] (P1) to node[above] {} (P2);
\draw[->] (P3) to node[above] {} (P4);
\draw[->] (P5) to node[above] {} (P6);
\end{scope}

\begin{scope}[shift={(18,-4)}]
\node[left]  (P1) at (0, 0) {$\mathsf{X}$};
\node[right] (P2) at (0.75, 0) {$3\mathsf{X}$};
\node[left]  (P3) at (0,-1) {$\mathsf{X+Y}$};
\node[right] (P4) at (0.75,-1) {$3\mathsf{Y}$};
\node[left]  (P5) at (0,-2) {$\mathsf{Y+Z}$};
\node[right] (P6) at (0.75,-2) {$\mathsf{Z}$};
\draw[->] (P1) to node[above] {} (P2);
\draw[->] (P3) to node[above] {} (P4);
\draw[->] (P5) to node[above] {} (P6);
\end{scope}

\begin{scope}[shift={(0,-8)}]
\node[left]  (P1) at (0, 0) {$\mathsf{X+Z}$};
\node[right] (P2) at (0.75, 0) {$2\mathsf{X}+\mathsf{Z}$};
\node[left]  (P3) at (0,-1) {$\mathsf{X+Y}$};
\node[right] (P4) at (0.75,-1) {$2\mathsf{Y}$};
\node[left]  (P5) at (0,-2) {$\mathsf{Y}$};
\node[right] (P6) at (0.75,-2) {$\mathsf{0}$};
\draw[->] (P1) to node[above] {} (P2);
\draw[->] (P3) to node[above] {} (P4);
\draw[->] (P5) to node[above] {} (P6);
\end{scope}

\begin{scope}[shift={(0,-12)}]
\node[left]  (P1) at (0, 0) {$\mathsf{X+Z}$};
\node[right] (P2) at (0.75, 0) {$2\mathsf{X}+\mathsf{Z}$};
\node[left]  (P3) at (0,-1) {$\mathsf{X+Y}$};
\node[right] (P4) at (0.75,-1) {$3\mathsf{Y}$};
\node[left]  (P5) at (0,-2) {$\mathsf{Y}$};
\node[right] (P6) at (0.75,-2) {$\mathsf{0}$};
\draw[->] (P1) to node[above] {} (P2);
\draw[->] (P3) to node[above] {} (P4);
\draw[->] (P5) to node[above] {} (P6);
\end{scope}

\begin{scope}[shift={(6,-8)}]
\node[left]  (P1) at (0, 0) {$\mathsf{X+Z}$};
\node[right] (P2) at (0.75, 0) {$2\mathsf{X}+\mathsf{Z}$};
\node[left]  (P3) at (0,-1) {$\mathsf{X+Y}$};
\node[right] (P4) at (0.75,-1) {$2\mathsf{Y}$};
\node[left]  (P5) at (0,-2) {$\mathsf{Y+Z}$};
\node[right] (P6) at (0.75,-2) {$\mathsf{Z}$};
\draw[->] (P1) to node[above] {} (P2);
\draw[->] (P3) to node[above] {} (P4);
\draw[->] (P5) to node[above] {} (P6);
\end{scope}

\begin{scope}[shift={(6,-12)}]
\node[left]  (P1) at (0, 0) {$\mathsf{X+Z}$};
\node[right] (P2) at (0.75, 0) {$2\mathsf{X}+\mathsf{Z}$};
\node[left]  (P3) at (0,-1) {$\mathsf{X+Y}$};
\node[right] (P4) at (0.75,-1) {$3\mathsf{Y}$};
\node[left]  (P5) at (0,-2) {$\mathsf{Y+Z}$};
\node[right] (P6) at (0.75,-2) {$\mathsf{Z}$};
\draw[->] (P1) to node[above] {} (P2);
\draw[->] (P3) to node[above] {} (P4);
\draw[->] (P5) to node[above] {} (P6);
\end{scope}

\begin{scope}[shift={(12,-8)}]
\node[left]  (P1) at (0, 0) {$\mathsf{X+Z}$};
\node[right] (P2) at (0.75, 0) {$2\mathsf{X}+\mathsf{Z}$};
\node[left]  (P3) at (0,-1) {$\mathsf{X+Y}$};
\node[right] (P4) at (0.75,-1) {$2\mathsf{Y}$};
\node[left]  (P5) at (0,-2) {$\mathsf{Y+W}$};
\node[right] (P6) at (0.75,-2) {$\mathsf{W}$};
\draw[->] (P1) to node[above] {} (P2);
\draw[->] (P3) to node[above] {} (P4);
\draw[->] (P5) to node[above] {} (P6);
\end{scope}

\begin{scope}[shift={(12,-12)}]
\node[left]  (P1) at (0, 0) {$\mathsf{X+Z}$};
\node[right] (P2) at (0.75, 0) {$2\mathsf{X}+\mathsf{Z}$};
\node[left]  (P3) at (0,-1) {$\mathsf{X+Y}$};
\node[right] (P4) at (0.75,-1) {$3\mathsf{Y}$};
\node[left]  (P5) at (0,-2) {$\mathsf{Y+W}$};
\node[right] (P6) at (0.75,-2) {$\mathsf{W}$};
\draw[->] (P1) to node[above] {} (P2);
\draw[->] (P3) to node[above] {} (P4);
\draw[->] (P5) to node[above] {} (P6);
\end{scope}

\begin{scope}[shift={(18,-8)}]
\node[left]  (P1) at (0, 0) {$\mathsf{Z+X}$};
\node[right] (P2) at (0.75, 0) {$2\mathsf{X}$};
\node[left]  (P3) at (0,-1) {$\mathsf{X+Y}$};
\node[right] (P4) at (0.75,-1) {$2\mathsf{Y}$};
\node[left]  (P5) at (0,-2) {$\mathsf{Y+Z}$};
\node[right] (P6) at (0.75,-2) {$2\mathsf{Z}$};
\draw[->] (P1) to node[above] {} (P2);
\draw[->] (P3) to node[above] {} (P4);
\draw[->] (P5) to node[above] {} (P6);
\end{scope}

\begin{scope}[shift={(18,-12)}]
\node[left]  (P1) at (0, 0) {$2\mathsf{X}$};
\node[right] (P2) at (0.75, 0) {$3\mathsf{X}$};
\node[left]  (P3) at (0,-1) {$\mathsf{X+Y}$};
\node[right] (P4) at (0.75,-1) {$2\mathsf{Y}+\mathsf{Z}$};
\node[left]  (P5) at (0,-2) {$\mathsf{Y+Z}$};
\node[right] (P6) at (0.75,-2) {$\mathsf{0}$};
\draw[->] (P1) to node[above] {} (P2);
\draw[->] (P3) to node[above] {} (P4);
\draw[->] (P5) to node[above] {} (P6);
\end{scope}

\draw (21.5,-7)--(15.5,-7)--(15.5,-15);
\draw (15.5,-11)--(21.5,-11);
\draw (-2.5,1)--(21.5,1)--(21.5,-15)--(-2.5,-15)--(-2.5,1);

\end{tikzpicture}
\end{center}
\caption{The list of all three-reaction, quadratic, trimolecular $(n,3,2)$ networks whose mass-action system has a periodic orbit for some rate constants. There are sixteen such networks. Four are members of the family \eqref{eq:case_7}, eight are derived from these by adding a trivial species, and two are obtained by adding two trivial species. The latter two are the only ones with four species. The Ivanova reactions and the Lifted LVA complete the list. Notice that the only ones that are bimolecular are the Lotka reactions \eqref{eq:lotka_only} and the Ivanova reactions \eqref{eq:ivanova_only}.}
\label{fig:16_networks}
\end{figure}

\subsection{Networks with the reaction \texorpdfstring{$\mathsf{2X} \to \mathsf{3X}$}{}} \label{subsec:2X_3X}

In this subsection, we prove three lemmas about quadratic $(n,3,2)$ networks without trivial species that include the reaction $\mathsf{2X} \to \mathsf{3X}$; these results are used in the proof of \Cref{thm:ndim_trimolec} above. In \Cref{lem:saddle_or_lifted_LVA} we discuss the case $n=3$ and show that the only networks which lead to a periodic orbit are those in the Lifted LVA family \eqref{eq:LVA_d_lifted}. In \Cref{lem:4species_saddle} we prove that when $n=4$ any positive equilibrium must be a saddle. Finally, in \Cref{lem:5species} we find that no network admits a positive equilibrium when $n\geq5$.

\begin{lemma}\label{lem:saddle_or_lifted_LVA}
Suppose a quadratic $(3,3,2)$ mass-action system with no trivial species includes the reaction $2\mathsf{X} \to 3\mathsf{X}$ and has a periodic orbit for some rate constants. Then the network is a member of the Lifted LVA family \eqref{eq:LVA_d_lifted}.
\end{lemma}
\begin{proof}
Consider a network satisfying the assumptions of the lemma. Any periodic orbit must be positive and the network must be dynamically nontrivial (see \Cref{subsec:saddle_no_periodic}); consequently each species is gained in at least one reaction and is lost in at least one reaction. In $\mathsf{2X} \to \mathsf{3X}$, an $\mathsf{X}$ is gained. W.l.o.g.\ assume that $\mathsf{X}$ is lost in the second reaction. Then the source of that reaction is either $\mathsf{X}$ or $\mathsf{X+Y}$ (note that the second source cannot be $2\mathsf{X}$ by \Cref{lem:sources_on_a_line}). Moreover, the target of the reaction does not include $\mathsf{X}$.

Suppose that the source of the second reaction is $\mathsf{X}$. As the network is dynamically nontrivial, the third reaction must be $\mathsf{Y+Z} \to a \mathsf{X}$ (for some $a\geq0$). To guarantee that the rank of the network is two, the second reaction must then be $\mathsf{X} \to b\mathsf{Y}+b\mathsf{Z}$ (for some $b\geq1$). As $\ker \Gamma$ is spanned by $[1-ab,1,1]^\top$, the network is dynamically nontrivial if and only if $a=0$. A short computation shows that the reduced Jacobian determinant \eqref{eq:detJred} at a positive equilibrium $(x,y,z)$ 
equals $-\frac{b}{xy}-\frac{b}{xz}$, which is negative. Thus, any positive equilibrium is a saddle, contradicting the occurrence of a periodic orbit  (see \Cref{subsec:saddle_no_periodic}). 

Suppose now that the source of the second reaction is $\mathsf{X+Y}$ and additionally suppose that $\mathsf{Y}$ is lost in that reaction. Then $\mathsf{Y}$ is gained in the third reaction. Further, $\mathsf{Z}$ must be gained in the second reaction and lost in the third one. The general form of the network is
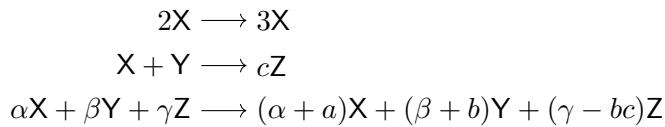
\begin{align*}
\begin{tikzpicture}[scale=0.6]

\node[left]  (P1) at (0, 0) {$2\mathsf{X}$};
\node[right] (P2) at (1, 0) {$3\mathsf{X}$};
\node[left]  (P3) at (0,-1) {$\mathsf{X+Y}$};
\node[right] (P4) at (1,-1) {$c \mathsf{Z}$};
\node[left]  (P5) at (0,-2) {$\alpha \mathsf{X} +\beta \mathsf{Y} + \gamma \mathsf{Z}$};
\node[right] (P6) at (1,-2) {$(\alpha+a) \mathsf{X}  + (\beta + b) \mathsf{Y} + (\gamma - bc) \mathsf{Z}$};

\draw[->] (P1) to node[above] {} (P2);
\draw[->] (P3) to node[above] {} (P4);
\draw[->] (P5) to node[above] {} (P6);

\end{tikzpicture}
\end{align*}
where the stoichiometric coefficient $\gamma-bc$ is taken to ensure $\rank \Gamma=2$. The parameters satisfy
\begin{align*}
    b\geq1, c\geq1, \gamma\geq1, \alpha + \beta \leq1.
\end{align*}
Examining the stoichiometric matrix, we find that the network is dynamically nontrivial if and only if $b-a>0$ (recall that $b\geq1$). A short calculation shows that the reduced Jacobian determinant \eqref{eq:detJred} at a positive equilibrium $(x,y,z)$ equals
\begin{align*}
    -(b-a)b\left(\frac{2-\alpha-\beta}{xy} + \frac{ c \gamma}{xz}\right),
\end{align*}
which is negative, i.e., the equilibrium is a saddle. This contradicts the occurrence of a periodic orbit (see \Cref{subsec:saddle_no_periodic}).

Finally, suppose that the source of the second reaction is $\mathsf{X+Y}$ and additionally suppose that $\mathsf{Y}$ is gained in that reaction. Taking also into account that the network is dynamically nontrivial and has rank two, the network must belong to the Lifted LVA family \eqref{eq:LVA_d_lifted}.
\end{proof}

\begin{lemma}\label{lem:4species_saddle}
Suppose a quadratic $(4,3,2)$ mass-action network with no trivial species contains the reaction $\mathsf{2X} \to \mathsf{3X}$ and has a positive equilibrium. Then the network is
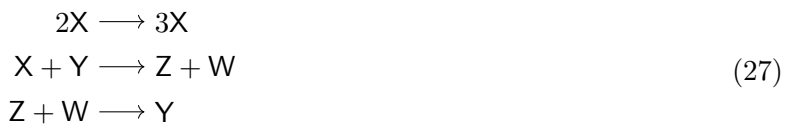
\begin{align}\label{eq:4species_saddle}
\begin{aligned}
\begin{tikzpicture}[scale=0.6]

\node[left]  (P1) at (0, 0) {$2\mathsf{X}$};
\node[right] (P2) at (1, 0) {$3\mathsf{X}$};
\node[left]  (P3) at (0,-1) {$\mathsf{X+Y}$};
\node[right] (P4) at (1,-1) {$\mathsf{Z+W}$};
\node[left]  (P5) at (0,-2) {$\mathsf{Z+W}$};
\node[right] (P6) at (1,-2) {$\mathsf{Y}$};

\draw[->] (P1) to node[above] {} (P2);
\draw[->] (P3) to node[above] {} (P4);
\draw[->] (P5) to node[above] {} (P6);

\end{tikzpicture}
\end{aligned}
\end{align}
For any choice of rate constants, the corresponding mass-action system has exactly one positive equilibrium in every positive stoichiometric class, and this equilibrium is a saddle. Consequently, the system admits no periodic orbit.
\end{lemma}
\begin{proof}
Consider a network satisfying the hypotheses of the lemma. In $\mathsf{2X} \to \mathsf{3X}$, species $\mathsf{X}$ is gained and therefore, as the network admits positive equilibria and hence must be dynamically nontrivial, there has to be a reaction where $\mathsf{X}$ is lost. The source of that reaction must be either $\mathsf{X}$, $2\mathsf{X}$, or $\mathsf{X+Y}$.

In case the source of the second reaction is $\mathsf{X}$ or $2\mathsf{X}$, the species $\mathsf{Y}$, $\mathsf{Z}$, $\mathsf{W}$ can only be gained there. However, then each of $\mathsf{Y}$, $\mathsf{Z}$, $\mathsf{W}$ must be lost in the third reaction, which is impossible in a quadratic network.

In case the source of the second reaction is $\mathsf{X+Y}$, the species $\mathsf{Z}$ and $\mathsf{W}$ can be lost only in the third reaction, the source of that reaction must be $\mathsf{Z+W}$. The general form of the network is then
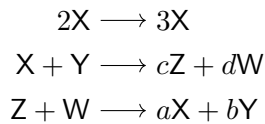
\begin{align*}
\begin{tikzpicture}[scale=0.6]

\node[left]  (P1) at (0, 0) {$2\mathsf{X}$};
\node[right] (P2) at (1, 0) {$3\mathsf{X}$};
\node[left]  (P3) at (0,-1) {$\mathsf{X+Y}$};
\node[right] (P4) at (1,-1) {$c\mathsf{Z} + d\mathsf{W}$};
\node[left]  (P5) at (0,-2) {$\mathsf{Z} + \mathsf{W}$};
\node[right] (P6) at (1,-2) {$a\mathsf{X}+b\mathsf{Y}$};

\draw[->] (P1) to node[above] {} (P2);
\draw[->] (P3) to node[above] {} (P4);
\draw[->] (P5) to node[above] {} (P6);

\end{tikzpicture}
\end{align*}
with $a, b, c, d\geq0$. We easily find that the network is dynamically nontrivial if and only if $a=0$ and $b=c=d=1$, which gives the network \eqref{eq:4species_saddle}. By a short calculation, this network has a unique positive equilibrium in every positive stoichiometric class. Notice that network \eqref{eq:4species_saddle} was used in \Cref{subsec:detJred} as an illustrative example, where we concluded that any positive equilibrium is a saddle. Hence, by the observations in \Cref{subsec:saddle_no_periodic}, the system admits no periodic orbit.
\end{proof}


\begin{lemma}\label{lem:5species}
For $n\geq5$, any $(n,3,2)$ network without trivial species, and including the reaction $\mathsf{2X} \to \mathsf{3X}$, must be dynamically trivial. Consequently, the corresponding mass-action system admits no periodic orbit. 
\end{lemma}
\begin{proof}
As at the beginning of the proofs of \Cref{lem:saddle_or_lifted_LVA,lem:4species_saddle}, the sources of the first and second reactions can involve only two species in total. Also, because of the bimolecularity of the sources, the third reaction can also have at most two different species in its source. Hence at least one of the species does not appear in any of the sources, and its concentration is nondecreasing. Since it is, by assumption, not a trivial species, the network is dynamically trivial and admits no periodic orbit (see \Cref{subsec:saddle_no_periodic}).
\end{proof}

\section{Conclusions} \label{sec:conclusions}

In this paper, we have studied quadratic networks with three reactions, and identified all trimolecular mass-action systems in this class which admit a periodic orbit. In all of these systems, any nearby orbits of a periodic orbit are periodic, too. Thus, the existence of an isolated periodic orbit in a quadratic mass-action system with three reactions implies that at least one target complex has a molecularity of four or more. In fact, in the two-species case (with no assumption on the target molecularities), we classified all networks that admit a nondegenerate Andronov--Hopf bifurcation, and thus, a limit cycle, see cases \boxed{9} and \boxed{10} in \Cref{thm:ten_src_triples}. However, we leave it open whether the networks in cases \boxed{9} and \boxed{10} admitting no Andronov--Hopf bifurcations can have a limit cycle. For example, we do not know whether the system
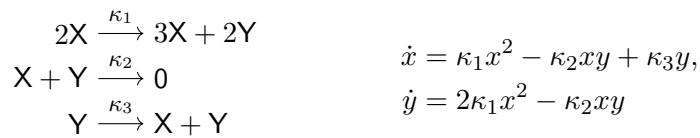
\begin{align*}
\begin{tikzpicture}[scale=0.6]

\node[left]  (P1) at (0, 0) {$2\mathsf{X}$};
\node[right] (P2) at (1, 0) {$3\mathsf{X}+2\mathsf{Y}$};
\node[left]  (P3) at (0,-1) {$\mathsf{X+Y}$};
\node[right] (P4) at (1,-1) {$\mathsf{0}$};
\node[left]  (P5) at (0,-2) {$\mathsf{Y}$};
\node[right] (P6) at (1,-2) {$\mathsf{X}+\mathsf{Y}$};

\draw[->] (P1) to node[above] {\footnotesize $\kappa_1$} (P2);
\draw[->] (P3) to node[above] {\footnotesize $\kappa_2$} (P4);
\draw[->] (P5) to node[above] {\footnotesize $\kappa_3$} (P6);

\node at (10,-1)
{$\begin{aligned}
\dot{x} &= \kappa_1 x^2  - \kappa_2 xy + \kappa_3 y, \\
\dot{y} &= 2\kappa_1 x^2 - \kappa_2 xy \\
\end{aligned}$};

\end{tikzpicture}
\end{align*}
admits a limit cycle.

The networks in cases \boxed{9} and \boxed{10} in \Cref{thm:ten_src_triples} that admit a nondegenerate Andronov--Hopf bifurcation can be enlarged by adding a new (nontrivial) species in a way that the rank of the network is preserved \cite{banaji:boros:hofbauer:2022}. The resulting quadratic $(3,3,2)$ networks admit a supercritical Andronov--Hopf bifurcation by \cite[Remark 6]{banaji:boros:hofbauer:2022}. For example, the networks
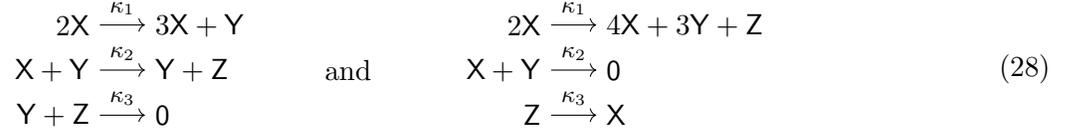
\begin{align}\label{eq:lifted_tetra_hepta}
\begin{aligned}
\begin{tikzpicture}[scale=0.6]

\node[left]  (P1) at (0, 0) {$2\mathsf{X}$};
\node[right] (P2) at (1, 0) {$3\mathsf{X}+\mathsf{Y}$};
\node[left]  (P3) at (0,-1) {$\mathsf{X+Y}$};
\node[right] (P4) at (1,-1) {$\mathsf{Y+Z}$};
\node[left]  (P5) at (0,-2) {$\mathsf{Y+Z}$};
\node[right] (P6) at (1,-2) {$\mathsf{0}$};

\draw[->] (P1) to node[above] {\footnotesize $\kappa_1$} (P2);
\draw[->] (P3) to node[above] {\footnotesize $\kappa_2$} (P4);
\draw[->] (P5) to node[above] {\footnotesize $\kappa_3$} (P6);


\node at (5.5,-1) {and};

\begin{scope}[shift={(10,0)}]
\node[left]  (P1) at (0, 0) {$2\mathsf{X}$};
\node[right] (P2) at (1, 0) {$4\mathsf{X}+3\mathsf{Y}+\mathsf{Z}$};
\node[left]  (P3) at (0,-1) {$\mathsf{X+Y}$};
\node[right] (P4) at (1,-1) {$\mathsf{0}$};
\node[left]  (P5) at (0,-2) {$\mathsf{Z}$};
\node[right] (P6) at (1,-2) {$\mathsf{X}$};

\draw[->] (P1) to node[above] {\footnotesize $\kappa_1$} (P2);
\draw[->] (P3) to node[above] {\footnotesize $\kappa_2$} (P4);
\draw[->] (P5) to node[above] {\footnotesize $\kappa_3$} (P6);

\end{scope}

\end{tikzpicture}
\end{aligned}
\end{align}
obtained from \eqref{eq:tetra} with $d=0$ and \eqref{eq:hepta} with $(c,d)=(1,0)$, respectively, both admit a supercritical Andronov--Hopf bifurcation, and thus, a stable limit cycle. We leave it open whether there are quadratic $(3,3,2)$ networks (with no assumption on the target molecularities) that admit a nondegenerate Andronov--Hopf bifurcation which is not inherited from a smaller network as in these examples.

Interestingly, the octomolecular network on the right of \eqref{eq:lifted_tetra_hepta} admits not only a supercritical Andronov--Hopf bifurcation, but also a Bogdanov--Takens bifurcation \cite[Section 8.4]{kuznetsov:2004}, and hence, a homoclinic bifurcation. The stoichiometric classes are given by $x-y+z = C$ for $C\in\mathbb{R}$, and the set of positive equilibria is the curve $\left\{\left(t,\frac{3\kappa_1}{\kappa_2}t,\frac{\kappa_1}{\kappa_3}t^2\right)\colon t>0\right\}$, which intersects the stoichiometric classes in $0$, $1$, or $2$ points. One may verify that for fixed $\kappa_2>0$ and $C<0$, a supercritical Bogdanov--Takens bifurcation occurs at
\begin{align*}
(\kappa_1,\kappa_3) = \kappa_2 \left(\frac{3+\sqrt{6}}{3},\frac{-2C}{3+\sqrt{6}}\right).
\end{align*}
We conclude this paragraph with the observation that the mass-action differential equation of the octomolecular network in question is identical to that of
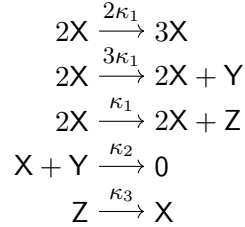
\begin{align*}
\begin{tikzpicture}[scale=0.6]

\node[left]  (P1)  at (0, 0) {$2\mathsf{X}$};
\node[right] (P2)  at (1, 0) {$3\mathsf{X}$};
\node[left]  (P3)  at (0,-1) {$2\mathsf{X}$};
\node[right] (P4)  at (1,-1) {$2\mathsf{X}+\mathsf{Y}$};
\node[left]  (P5)  at (0,-2) {$2\mathsf{X}$};
\node[right] (P6)  at (1,-2) {$2\mathsf{X}+\mathsf{Z}$};
\node[left]  (P7)  at (0,-3) {$\mathsf{X+Y}$};
\node[right] (P8)  at (1,-3) {$\mathsf{0}$};
\node[left]  (P9)  at (0,-4) {$\mathsf{Z}$};
\node[right] (P10) at (1,-4) {$\mathsf{X}$};

\draw[->] (P1) to node[above] {\footnotesize $2\kappa_1$} (P2);
\draw[->] (P3) to node[above] {\footnotesize $3\kappa_1$} (P4);
\draw[->] (P5) to node[above] {\footnotesize $\kappa_1$} (P6);
\draw[->] (P7) to node[above] {\footnotesize $\kappa_2$} (P8);
\draw[->] (P9) to node[above] {\footnotesize $\kappa_3$} (P10);

\end{tikzpicture}
\end{align*}
which is a five-reaction, trimolecular system, with restrictions on its rate constants.

The construction at the end of the previous paragraph works in general: the mass-action differential equation of any quadratic network (with no assumption on the target molecularities) is identical to that of some trimolecular, quadratic network with some restrictions on its rate constants. Thus, claims about systems with high target molecularity can often be reduced to claims about trimolecular systems, at the cost of increasing the total number of reactions. Indeed, given any mass-action system, to obtain an equivalent system with trimolecular targets we may replace each reaction of the form
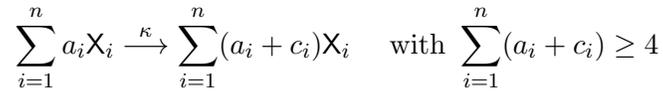
\begin{align*}
\sum_{i=1}^na_i\mathsf{X}_i \stackrel{\kappa}{\longrightarrow} \sum_{i=1}^n(a_i+c_i)\mathsf{X}_i \quad \text{ with }\sum_{i=1}^n(a_i+c_i)\geq 4
\end{align*}
by the (at most) $n$ reactions
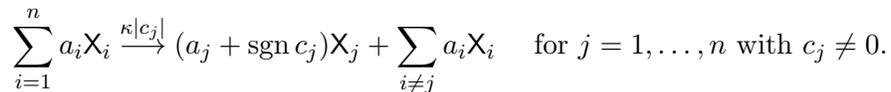
\begin{align*}
\sum_{i=1}^na_i\mathsf{X}_i \stackrel{\kappa|c_j|}{\longrightarrow} (a_j+\sgn c_j)\mathsf{X}_j + \sum_{i\neq j}a_i\mathsf{X}_i \quad \text{ for } j=1,\ldots,n \text{ with }c_j\neq0.
\end{align*}
For example, the three-reaction, tetramolecular network \eqref{eq:tetra_simplest} gives rise to the same mass-action differential equation as
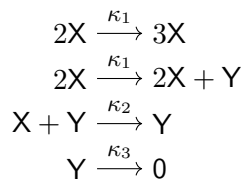
\begin{align*}
\begin{tikzpicture}[scale=0.6]

\node[left]  (P1) at (0, 0) {$2\mathsf{X}$};
\node[right] (P2) at (1, 0) {$3\mathsf{X}$};
\node[left]  (P3) at (0,-1) {$2\mathsf{X}$};
\node[right] (P4) at (1,-1) {$2\mathsf{X}+\mathsf{Y}$};
\node[left]  (P5) at (0,-2) {$\mathsf{X+Y}$};
\node[right] (P6) at (1,-2) {$\mathsf{Y}$};
\node[left]  (P7) at (0,-3) {$\mathsf{Y}$};
\node[right] (P8) at (1,-3) {$\mathsf{0}$};

\draw[->] (P1) to node[above] {\footnotesize $\kappa_1$} (P2);
\draw[->] (P3) to node[above] {\footnotesize $\kappa_1$} (P4);
\draw[->] (P5) to node[above] {\footnotesize $\kappa_2$} (P6);
\draw[->] (P7) to node[above] {\footnotesize $\kappa_3$} (P8);

\end{tikzpicture}
\end{align*}
which is trimolecular, has one more reaction, and has a restriction on its rate constants. As this example and the first network in \Cref{subsec:counterex} show, there exist quadratic, trimolecular $(2,4,2)$ networks admitting a nondegenerate Andronov--Hopf bifurcation with mass-action kinetics. In future work, we plan to find all such networks, and identify which of the corresponding mass-action systems admit Bogdanov--Takens bifurcation.


\bibliographystyle{abbrv}
\bibliography{biblio}

\end{document}